\documentclass[reqno]{amsart}

\usepackage[utf8]{inputenc} 
\usepackage{amssymb}
\usepackage{graphicx}
\usepackage{latexsym}
\usepackage{stmaryrd}
\usepackage{enumitem}
\usepackage[bookmarks,hyperfootnotes=false]{hyperref}
\usepackage{multirow}
\usepackage{caption}
\usepackage{xcolor}
\colorlet{darkishRed}{red!80!black}
\colorlet{darkishBlue}{blue!60!black}
\colorlet{darkishGreen}{green!60!black}
\usepackage{hyperref}
\hypersetup{
    draft = false,
    bookmarksopen=true,
    colorlinks,
    linkcolor={red!60!black},
    citecolor={green!60!black},
    urlcolor={blue!60!black}
}

\usepackage{tikz}
\usepackage{tikz-cd}
\usetikzlibrary{calc,through,intersections,arrows, trees, positioning, decorations.pathmorphing, cd}
\usepackage{relsize}
\usepackage{comment}
\usepackage{xifthen}
\usepackage{mathrsfs}
\usepackage{svg}
\usepackage{mathtools}
\usepackage{nccmath}
\usepackage{pifont}

\newcommand{\Abs}[1]{\partial_{\Omega} {#1}}
\newcommand{\AbsC}[1]{\partial_{\Gamma} {#1}}
\newcommand{\crit}{\normalfont\text{crit}}

\newcommand{\CX}{\breve{\cC}_X}

\newcommand{\CC}[1]{\breve{\cC}_{#1}}
\newcommand{\invLim}{\varprojlim}
\newcommand{\cb}{\mathfrak{c}}

\newcommand{\rest}{\upharpoonright}

\newcommand{\rsep}[2]{({#1},{#2})}
\newcommand{\lsep}[2]{({#1},{#2})}
\newcommand{\sep}[2]{\{{#1},{#2}\}}

\newcommand{\crsep}[1]{\rsep{#1}{\CC{#1}}}

\renewcommand{\subset}{\subseteq}
\renewcommand{\supset}{\supseteq}

\newcommand\xqed[1]{%
  \leavevmode\unskip\penalty9999 \hbox{}\nobreak\hfill
  \quad\hbox{#1}}

\def\comp{com\-pac\-ti\-fi\-ca\-tion}

\def\SC{Stone-Čech}

\def\TFAD{Let $G$ be any connected graph and let $U\subset V(G)$ be any vertex set. Then the following assertions are complementary:}
\def\tfad{the following assertions are complementary:}
\newcommand{\at}{attached to }

\newcommand{ \N } { \mathbb{N} }

\newcommand{\dblue}[1]{\textcolor{darkishBlue}{#1}}
\newcommand{\dc}[1]{\lceil #1\rceil}
\newcommand{\uc}[1]{\lfloor #1\rfloor}
\newcommand{\guc}[1]{\lfloor\mkern-1.4\thinmuskip\lfloor #1\rfloor\mkern-1.4\thinmuskip\rfloor}
\newcommand{\nt}{T_{\textup{\textsc{nt}}}}

\makeatletter

\def\calCommandfactory#1{%
   \expandafter\def\csname c#1\endcsname{\mathcal{#1}}}
\def\frakCommandfactory#1{%
   \expandafter\def\csname frak#1\endcsname{\mathfrak{#1}}}

\newcounter{ctr}
\loop
  \stepcounter{ctr}
  \edef\X{\@Alph\c@ctr}
  \expandafter\calCommandfactory\X
  \expandafter\frakCommandfactory\X
  \edef\Y{\@alph\c@ctr}
  \expandafter\frakCommandfactory\Y
\ifnum\thectr<26
\repeat

\renewcommand{\cC}{\mathscr{C}}
\renewcommand{\cD}{\mathscr{D}}
\renewcommand{\cK}{\mathscr{K}}

\setenumerate{label={\normalfont (\roman*)},itemsep=0pt}

\def\lowfwd #1#2#3{{\mathop{\kern0pt #1}\limits^{\kern#2pt\raise.#3ex
\vbox to 0pt{\hbox{$\scriptscriptstyle\rightarrow$}\vss}}}}
\def\lowbkwd #1#2#3{{\mathop{\kern0pt #1}\limits^{\kern#2pt\raise.#3ex
\vbox to 0pt{\hbox{$\scriptscriptstyle\leftarrow$}\vss}}}}
\def\fwd #1#2{{\lowfwd{#1}{#2}{15}}}
\def\Sinf{S_{\aleph_0}}

\def\vS{{\hskip-1pt{\fwd S3}\hskip-1pt}}
\def\vSinf{\vS_{\mkern-.85\thinmuskip\aleph_0}}

\def\vE{{\hskip-1pt{\fwd{E}{3.5}}\hskip-1pt}}
\def\vF{{\hskip-1pt{\fwd{F}{3.5}}\hskip-1pt}}
\def\ve{\kern-1.5pt\lowfwd e{1.5}2\kern-1pt}
\def\ev{\kern-1pt\lowbkwd e{0.5}2\kern-1pt}
\def\vf{\kern-2pt\lowfwd f{2.5}2\kern-1pt}

\newtheorem{theorem}{Theorem}[section] 
\newtheorem{proposition}[theorem]{Proposition}
\newtheorem{corollary}[theorem]{Corollary}
\newtheorem{lemma}[theorem]{Lemma}

\newtheorem{problem}[theorem]{Problem}
\newtheorem{mainresult}{Theorem}

\newtheorem*{starCombIntro}{Star-comb lemma}

\newtheorem{innercustomthm}{Theorem}

\newenvironment{customthm}[1]
  {\renewcommand\theinnercustomthm{#1}\innercustomthm}
  {\endinnercustomthm}

\theoremstyle{definition}
\newtheorem{example}[theorem]{Example}

\newtheorem{definition}[theorem]{Definition}

\theoremstyle{remark}

\newtheorem*{ack}{Acknowledgement}

\begin{document}

\title[Duality theorems for stars and combs]{Duality theorems for stars and combs\\
I: Arbitrary stars and combs}

\author{Carl Bürger}
\author{Jan Kurkofka}
\address{Universität Hamburg, Department of Mathematics, Bundesstraße 55 (Geomatikum), 20146 Hamburg, Germany}
\email{carl.buerger@uni-hamburg.de, jan.kurkofka@uni-hamburg.de}


\@namedef{subjclassname@2020}{\textup{2020} Mathematics Subject Classification}

\begin{abstract}
Extending the well-known star-comb lemma for infinite graphs, we characterise the graphs that do not contain an infinite comb or an infinite star, respectively, attached to a given set of vertices.
We offer several characterisations: in terms of normal trees, tree-de\-com\-po\-si\-tions, ranks of rayless graphs and tangle-distinguishing separators.
\end{abstract}
\vspace*{-1.14cm} 
\maketitle

\vspace*{-.75cm}

\section{Introduction}

\noindent It is well known, and easy to see, that every finite connected graph contains either a long path or a vertex of high degree.
Similarly,
\vspace{.5\baselineskip}
\begin{fleqn}%
\begin{equation*}%
\tag{$\ast$}%
\hspace{\parindent}\begin{aligned}\label{Fact:infGhasEitherRayOrInfDefVx}%
    \parbox{\textwidth-3\parindent}{\emph{Every infinite connected graph contains either a ray or a vertex of\\infinite degree}}%
\end{aligned}%
\end{equation*}%
\end{fleqn}%

\vspace{.5\baselineskip}
\noindent \cite[Proposition~8.2.1]{DiestelBook5}. 
Here, a \emph{ray} is a one-way infinite path. 
Call two properties of infinite graphs \emph{dual}, or \emph{complementary}, in a class of infinite graphs if they partition that class.
Despite~(\ref*{Fact:infGhasEitherRayOrInfDefVx}), the two properties of 
 `containing a ray' and `containing a vertex of infinite degree' are not complementary in the class of all infinite graphs: 
 an infinite complete graph, for example, contains both. 
Hence it is natural to ask for structures, more specific than vertices of infinite degree and rays, whose existence is complementary to that of rays and vertices of infinite degree, respectively.
Such structures do indeed exist.

For example, the property of having a vertex of infinite degree is trivially complementary, for connected infinite graphs, to the property that all distance classes from any fixed vertex are finite.
This duality is employed to prove~(\ref*{Fact:infGhasEitherRayOrInfDefVx}): if all the distance classes from some vertex are finite, then applying K\H{o}nig's infinity lemma \mbox{\cite[Lemma~8.1.2]{DiestelBook5}} to these classes yields a~ray.

Similarly, it is easy to see that having a Schmidt rank is complementary for infinite graphs to containing a ray~\cite{Schmidt1983}; see Section~\ref{subsection:CombRank} for the definition of the Schmidt rank. 
This duality allows for an alternative proof of~(\ref*{Fact:infGhasEitherRayOrInfDefVx}), as follows.
If $G$ is rayless, connected and infinite, then it has some rank $\alpha>0$.
Hence there is a finite vertex set $X\subset V(G)$ such that every component of $G-X$ has rank~$<\alpha$.
Then $G-X$ must have infinitely many components, and so by the pigeonhole principle some vertex in $X$ has infinite degree in $G$.

A stronger and localised version of~(\ref*{Fact:infGhasEitherRayOrInfDefVx}) is the star-comb lemma \cite[Lemma~8.2.2]{DiestelBook5}, a standard tool in infinite graph theory. 
Recall that a \emph{comb} is the union of a ray $R$ (the comb's \emph{spine}) with infinitely many disjoint finite paths, possibly trivial, that have precisely their first vertex on~$R$. 
The last vertices of those paths are the \emph{teeth} of this comb.
Given a vertex set $U$, a \emph{comb attached to} $U$ is a comb with all its teeth in $U$, and a \emph{star attached to} $U$ is a subdivided infinite star with all its leaves in $U$.
Then the set of teeth is the \emph{attachment set} of the comb, and the set of leaves is the \emph{attachment set} of the star.
\begin{starCombIntro}
Let $U$ be an infinite set of vertices in a connected graph $G$.
Then $G$ contains either a comb 
\at $U$ or a star \at $U$.
\end{starCombIntro}

Although the star-comb lemma trivially implies assertion~(\ref*{Fact:infGhasEitherRayOrInfDefVx}), with $U:=V(G)$, it is not primarily about the existence of one subgraph or another.
Rather, it tells us something about the nature of connectedness in infinite graphs: that the way in which they link up their infinite sets of vertices can take two fundamentally different forms, a star and a comb.
These two possibilities apply separately to all their infinite sets $U$ of vertices, and clearly, the smaller $U$ the stronger the assertion.

Just like the existence of rays or vertices of infinite degree, the existence of stars or combs \at a given set $U$ is not complementary (in the class of all infinite connected graphs containing $U$). 
In this paper, we determine
structures that are complementary to stars, and structures that are complementary to combs (always with respect to a fixed set $U$).

As stars and combs can interact with each other, this is not the end of the story.
For example, a given set $U$ might be connected in $G$ by both a star and a comb, even with infinitely intersecting sets of leaves and teeth. 
To formalise this, let us say that a subdivided star $S$ \emph{dominates} a comb $C$ if infinitely many of the leaves of $S$ are also teeth of $C$.
A \emph{dominating star} in a graph $G$ then is a subdivided star $S\subset G$ that dominates some comb $C\subset G$; and a \emph{dominated comb} in $G$ is a comb $C\subset G$ that is dominated by some subdivided star $S\subset G$.
In the remaining three papers~\cite{StarComb2TheDominatedComb,StarComb3TheUndominatedComb,StarComb4TheUndominatingStar} of this series we shall find complementary structures to the existence of these substructures (again, with respect to some fixed set $U$).


Just like the original star-comb lemma, our results can be applied as structural tools in other contexts.
Examples of such applications can be found in parts~\textsc{i--iii} of our series. 

\subsection*{I: Arbitrary stars and combs}


In this paper we prove five duality theorems for combs, and two for stars.
The complementary structures they offer are quite different, and not obviously interderivable.

Our first result is obtained by techniques of Jung~\cite{jung69}.
Recall that a rooted tree $T\subset G$ is \emph{normal} in $G$ if the endvertices of every $T$-path in $G$ are comparable in the tree-order of $T$, cf.~\cite{DiestelBook5}.

\begin{customthm}{\ref{CombTreeDualityI}}
\TFAD
\begin{enumerate}
\item $G$ contains a \dblue{comb} \at $U$;
\item there is a \dblue{rayless normal tree} $T\subset G$ that contains $U$.
\end{enumerate}
\end{customthm}

\noindent To see that (ii) implies that $G$---in fact, the normal tree $T$---contains a star \at $U$ when $U$ is infinite, pick from among the nodes of $T$ that lie below infinitely many vertices of $T$ in $U$ one that is maximal in the tree-order of $T$.
Then its up-closure in $T$ contains the desired star.

Even though the normal tree from (ii) is in general not spanning, its separation properties still tell us a lot about the ambient graph $G$.
Our next result captures this overall structure of $G$ more explicitly (refer to~\cite{DiestelBook5} for the definition of tree-decompositions and adhesion sets):

\newpage
\begin{customthm}{\ref{CombTreeDualityII}}
\TFAD
\begin{enumerate}
    \item $G$ contains a \dblue{comb} \at $U$;
    \item $G$ has a \dblue{rayless tree-decomposition} into parts each containing at most finitely many vertices from $U$ and whose parts at non-leaves of the decomposition tree are all finite.
\end{enumerate}
Moreover, the tree-decomposition in \emph{(ii)} can be chosen with connected adhesion sets.
\end{customthm}

\noindent For $U=V(G)$, this theorem implies the following characterisation of rayless graphs by Halin~\cite{HalinRayless}: \emph{$G$ is rayless if and only if $G$ has a rayless tree-decomposition into finite parts.}

While Theorems~\ref{CombTreeDualityI} and~\ref{CombTreeDualityII} tell us about the structure of the graph around $U$, they further imply a more localised duality theorem for combs.
Call a finite vertex set $X\subset V(G)$ \emph{critical} if the collection $\CX$ of the components of $G-X$ having their neighbourhood precisely equal to $X$ is infinite.

\begin{customthm}{\ref{CombCritDuality}}
\TFAD
\begin{enumerate}
\item $G$ contains a \dblue{comb} \at $U$;
\item for every infinite $U'\subset U$ there is a \dblue{critical vertex set} $X\subset V(G)$ such that infinitely many of the components in $\CX$ meet $U'$.
\end{enumerate}
\end{customthm}

Critical vertex sets were introduced in~\cite{EndsTanglesCrit}.
As tangle-distinguishing separators, they have a surprising background involving the \SC\ \comp\ of~$G$, Robertson and Seymour's tangles from their graph-minor series, and Diestel's tangle \comp , cf.~\cite{StoneCechTangles,GMX,EndsAndTangles}.
Moreover, it turns out that Theorem~\ref{CombCritDuality} implies another characterisation of rayless graphs by Halin~\cite{Halin1965}.

The Schmidt rank of rayless graphs was employed by Bruhn, Diestel, Georgakopoulos and Sprüssel~\cite{UnfriendlyPartition} to prove the unfriendly partition conjecture for the class of rayless graphs by an involved transfinite induction on their rank.
We will show how the notion of a rank can be adapted to take into account a given set~$U$, so as to give a recursive definition of those graphs that do not contain a comb \at $U$. This yields our fourth duality theorem for combs:

\begin{customthm}{\ref{CombRankDuality}}
\TFAD
\begin{enumerate}
    \item $G$ contains a \dblue{comb} \at $U$;
    \item $G$ has a \dblue{$U$-rank}.
\end{enumerate}
\end{customthm}

With these four complementary structures for combs at hand, the question arises whether there is another complementary structure combining them all.
Our fifth duality theorem for combs shows that this is indeed possible:

\begin{customthm}{\ref{CombStrongDual}}
\TFAD
\begin{enumerate}
   \item $G$ contains a \dblue{comb} \at $U$;
   \item $G$ has a \dblue{tree-decomposition} that has the list $(\dagger)$ of properties.
\end{enumerate}
\end{customthm}
For the precise statement of this theorem, see Section~\ref{subsection:CombStrongDual}.
Essentially, the list $(\dagger)$ consists of the following four properties:\newpage
\begin{itemize}[label=\textbf{--}]
    \item its decomposition tree stems from a normal tree as in Theorem~\ref{CombTreeDualityI};
    \item it has the properties of the tree-decomposition in Theorem~\ref{CombTreeDualityII};
    \item the infinite-degree nodes of its decomposition tree correspond bijectively to the critical vertex sets of $G$ that are relevant in Theorem~\ref{CombCritDuality};
    \item the rank of its decomposition tree is equal to the $U$-rank of $G$\\from Theorem~\ref{CombRankDuality}.
\end{itemize}

Now that we have stated all the duality theorems for combs, let us turn to our two duality theorems for stars.
Recall that a vertex $v$ of $G$ \emph{dominates} a ray $R\subset G$ if there is an infinite $v$--$(R-v)$ fan in~$G$. 
Rays not dominated by any vertex are \emph{undominated}, cf.~\cite{DiestelBook5}.
Our first duality theorem for stars reads as follows:

\begin{customthm}{\ref{InitialStarNTduality}}
\TFAD
\begin{enumerate}
    \item $G$ contains a \dblue{star} \at $U$;
    \item there is a \dblue{locally finite normal tree} $T\subset G$ that contains $U$ and all whose rays are undominated in $G$.
\end{enumerate}
\end{customthm}

\noindent To see that (ii) implies that $G$---in fact, the normal tree---contains a comb \at $U$ when $U$ is infinite,
pick a ray in the locally finite down-closure of $U$ in the tree and extend it to a comb \at $U$.

We have seen normal trees before in our first duality theorem for combs, Theorem~\ref{CombTreeDualityI}.
Theorem~\ref{InitialStarNTduality} above compares with Theorem~\ref{CombTreeDualityI} as follows.
The only additional property required of the normal trees that are complementary to combs is that they are rayless.
Similarly, the normal trees that are complementary to stars have the additional property that they are locally finite.
However, they have the further property that all their rays are undominated in $G$.

This further property is necessary to ensure that the normal trees and stars in Theorem~\ref{InitialStarNTduality} exclude each other. To see this, let $G$ be obtained from a ray $R$ by completely joining its first vertex $r$ to all the other vertices of $R$, and suppose that $U=V(G)$.
Then $R\subset G$ with root $r$ is a locally finite normal tree containing $U$.
But the edges of $G$ at $r$ form a star \at $U$, so the further property is indeed necessary.

By contrast, we do not need to require in Theorem~\ref{CombTreeDualityI} that all the stars in the normal trees that are complementary to combs are undominating in $G$: this is already ensured by the nature of normal trees (see Lemma~\ref{NTlevelsDispersed} for details).

Our second duality theorem for stars is phrased in terms of tree-decompositions, similar to Theorem~\ref{CombTreeDualityII}:

\begin{customthm}{\ref{StarRayDualtiy}}
\TFAD
\begin{enumerate}
\item $G$ contains a \dblue{star} \at $U$;
\item $G$ has a \dblue{locally finite tree-decomposition} with finite and pairwise disjoint adhesion sets such that each part contains at most finitely many vertices from~$U$.
\end{enumerate}
Moreover, the tree-decomposition in \emph{(ii)} can be chosen with connected adhesion sets.
\end{customthm}

This paper is organised as follows.
Section~\ref{section:preliminaries} provides the tools and terminology that we use throughout this series.
Section~\ref{section:CombTreeCrit} and~\ref{section:StarRay} are dedicated to the duality theorems for combs and stars respectively.

Throughout this paper, $G=(V,E)$ is an arbitrary graph.

\begin{ack}
We are grateful to Nathan Bowler for many helpful comments, pointing out an error and simplifying some of our proofs.
We are grateful to our reviewers for more helpful comments that further improved this work.
\end{ack}

\section{Tools and terminology}\label{section:preliminaries}

\noindent Any graph-theoretic notation not explained here can be found in Diestel's textbook~\cite{DiestelBook5}.
A~non-trivial path $P$ is an $A$-\emph{path} for a set $A$ of vertices if $P$ has its endvertices but no inner vertex in $A$.
An independent set $M$ of edges in a graph $G$ is called a \emph{matching} of $A$ and $B$ for vertex sets $A,B\subset V(G)$ if every edge in $M$ has one endvertex in $A$ and the other in $B$.

\subsection{The star-comb lemma}
The predecessors of the star-comb lemma are the following facts:

\begin{lemma}[{\cite[Proposition~9.4.1]{DiestelBook5}}]\label{finiteGhasPathOrStar}
For every $m\in\N$ there is an $n\in\N$ such that each connected finite graph with at least $n$ vertices either contains a path of length $m$ or a star with $m$ leaves as a subgraph.
\end{lemma}

\begin{lemma}[{\cite[Proposition~8.2.1]{DiestelBook5}}]\label{infGhasRayOrStar}
A connected infinite graph contains either a ray or a vertex of infinite degree.
\end{lemma}

The latter is a direct consequence of the K\H{o}nig's infinity lemma, \cite[Lemma~8.1.2]{DiestelBook5}.
Lemma~\ref{finiteGhasPathOrStar} has been generalised to higher connectivity, \cite{GenGridTheorem,JoerisPhD,OporowskiOxleyThomas}, and so has Lemma~\ref{infGhasRayOrStar} in~\cite{GollinHeuerKcon,halin78,TypicalInfinitelyEdgeconnectedGraphs,FareyGraphChar,OporowskiOxleyThomas}.
For an overview we recommend the introduction of \cite{GollinHeuerKcon}.

For locally finite trees, Lemma~\ref{infGhasRayOrStar} already yields a comb:

\begin{lemma}\label{sclForLocFinTrees}
If $U$ is an infinite set of vertices in a locally finite rooted tree $T$, then $T$ contains a comb \at $U$ whose spine starts at the root.
\end{lemma}

\begin{proof}
The down-closure of $U$ in the tree-order of $T$ induces a locally finite subtree which, by Lemma~\ref{infGhasRayOrStar} above, contains a ray starting at the root, say.
This ray can be extended recursively  to the desired comb. 
\end{proof}

For rayless trees, the situation is simpler:

\begin{lemma}\label{sclForRaylessTrees}
If $U$ is an infinite set of vertices in a rayless rooted tree $T$, then $T$ contains a star \at $U$ which is contained in the up-closure of its central vertex in the tree-order of $T$.
\end{lemma}

\begin{proof}
Among all the nodes of $T$ that lie below some infinitely nodes from $U$, pick one node $t$, say, that is maximal in the tree-order of $T$.
Then $t$ has infinite degree and we find the desired star with centre $t$ in the up-closure of $t$.
\end{proof}

The general case can be reduced to trees:

\begin{lemma}[Star--comb lemma]\label{StarCombLemma}
Let $U$ be an infinite set of vertices in a connected graph~$G$.
Then $G$ contains either a comb \at ~$U$ or a star \at ~$U$.
\end{lemma}



\begin{proof}
Using Zorn's lemma we find a maximal tree $T\subset G$ all whose edges lie on a $U$-path in $T$. Then $T$ contains $U$.
If $T$ has a vertex $v$ of infinite degree, then its incident edges extend to $v$--$U$ paths whose union is the desired star \at ~$U$.
Otherwise $T$ is locally finite, and we find the desired comb \at ~$U$ using Lemma~\ref{sclForLocFinTrees} in~$T$.
\end{proof}

We wish to remark that the star-comb lemma has been generalised to take the cardinality of~$U$ into account, see the work by Diestel and Kühn~\cite{Ends} for the regular case and the work by Gollin and Heuer \cite[Corollary~8.1]{GollinHeuerKcon} for the singular case.





\subsection{Separations}
For a vertex set $X\subset V(G)$ we denote the collection of the components of $G-X$ by $\cC_X$.
If any $X\subset V(G)$ and $\cC\subset\cC_X$ are given, then these give rise to a separation of $G$ which we denote by
\begin{align*} 
    \sep{X}{\cC}:=\big\{\;V\setminus V[\cC]\;,\;X\cup V[\cC]\;\big\}
\end{align*}
where $V[\cC]=\bigcup\,\{\,V(C)\mid C\in\cC\,\}$.
Note that every separation $\{A,B\}$ of $G$ can be written in this way. 
For the orientations of $\sep{X}{\cC}$ we write
\begin{align*}
\rsep{X}{\cC}:=\big(\;V\setminus V[\cC]\;,\;X\cup V[\cC]\;\big)\quad\text{and}\quad\lsep{\cC}{X}:=\big(\;V[\cC]\cup X\;,\;V\setminus V[\cC]\;\big).
\end{align*}
We write $\sep{X}{C}$ and $\rsep{X}{C}$ and $\lsep{C}{X}$ instead of $\sep{X}{\{C\}}$ and $\rsep{X}{\{C\}}$ and $\lsep{\{C\}}{X}$ respectively.
The set of all finite-order separations of a graph $G$ is denoted by \mbox{$\Sinf=\Sinf(G)$}.

\subsection{Ends of graphs}

We write $\cX=\cX(G)$ for the collection of all finite subsets of the vertex set $V$ of $G$, partially ordered by inclusion.
An \emph{end} of $G$, as defined by Halin~\cite{halin64}, is an equivalence class of rays of $G$, where a ray is a one-way infinite path.
Here, two rays are said to be \emph{equivalent} if for every $X\in\cX$ both have a subray (also called \emph{tail}) in the same component of $G-X$. 
So in particular every end $\omega$ of $G$ chooses, for every $X\in\cX$, a unique component $C(X,\omega)=C_G(X,\omega)$ of $G-X$ in which every ray of $\omega$ has a tail. 
In this situation, the end $\omega$ is said to \emph{live} in $C(X,\omega)$.
The set of ends of a graph $G$ is denoted by $\Omega(G)$.
We use the convention that $\Omega$ always denotes the set of ends $\Omega(G)$ of the graph named $G$.

A vertex $v$ of $G$ \emph{dominates} a ray $R\subset G$ if there is an infinite $v$--$(R-v)$ fan in~$G$. 
Rays not dominated by any vertex are \emph{undominated}.
An end of $G$ is \emph{dominated} and \emph{undominated} if one (equivalently:~each) of its rays is dominated and undominated, respectively.
If $v$ does not dominate $\omega$, then there is an $X\in\cX$ which \emph{strictly separates} $v$ from $\omega$ in that $v\notin X\cup C(X,\omega)$.
More generally, if no vertex of $Y\in\cX$ dominates $\omega$, then there is an $X\in\cX$ \emph{strictly separating} $Y$ from $\omega$ in that $Y$ avoids the union $X\cup C(X,\omega)$.
Let us say that an oriented finite-order separation $(A,B)$ \emph{strictly separates} a set $X\subset V(G)$ of vertices from a set $\Psi\subset\Omega$ of ends if $X\subset A\setminus B$ and every end in $\Psi$ lives in a component of $G[B\setminus A]$.

Let us say that an end $\omega$ of $G$ is contained \emph{in the closure} of $M$, where $M$ is either  a subgraph of $G$ or a set of vertices of $G$, if for every $X\in\cX$ the component $C(X,\omega)$ meets $M$.
Equivalently, $\omega$ lies in the closure of $M$ if and only if $G$ contains a comb \at $M$ with its spine in $\omega$.
We write $\Abs{M}$ 
for the subset of $\Omega$ that consists of the ends of $G$ lying in the closure of $M$. 
Note that $\Abs{H}$ usually differs from $\Omega(H)$ for subgraphs $H\subseteq G$: 
For example, if $G$ is a ladder and $H$ is its outer double ray, then $\Abs{H}$ consists of the single end of $G$ while $\Omega(H)$ consists of the two ends of the double ray in $H$.
Readers familiar with $\vert G\vert$ as in \cite{DiestelBook5} will note that
$\Abs{M}$ is the intersection of $\Omega$ with the closure of $M$ in $\vert G\vert$, which in turn coincides with the topological frontier of $M\setminus\mathring{E}$ in the space $\vert G\vert\setminus\mathring{E}$.
If an end $\omega$ of $G$ does not lie in the closure of $M$, and if $X\in\cX$ witnesses this (in that $C(X,\omega)$ avoids $M$), then $X$ is said to \emph{separate} $\omega$ from $M$ (and $M$ from~$\omega$).
Carmesin~\cite{carmesin2014all} observed that

\begin{lemma}
Let $G$ be any graph.
If $H\subset G$ is a connected subgraph and $\omega$ is an undominated end of $G$ lying in the closure of $H$, then $H$ contains a ray from $\omega$.
\end{lemma}

\begin{proof}
Since $\omega$ lies in the closure of $H$ we find a comb in $G$ \at $H$ with spine in $\omega$. 
And as $\omega$ is undominated in $G$, the star-comb lemma in $H$ must return a comb in $H$ \at the attachment set of the first comb.
Then the two combs' spines are equivalent in $G$.
\end{proof}

Another way of viewing the ends of a graph goes via its \emph{directions}: choice maps $f$ assigning to every $X\in\cX$ a component of $G-X$ such that $f(X')\subset f(X)$ whenever $X'\supset X$. 
Every end $\omega$ defines a unique direction $f_\omega$ by mapping every $X\in \cX$ to $C(X,\omega)$. 
Conversely, Diestel and Kühn proved in~\cite{Ends} (Theorem~\ref{thm: directions} below) that every direction in fact comes from a unique end in this way, thus giving a one-to-one correspondence between the ends and the directions of a graph.

The advantage of this point of view stems from an inverse limit description of the directions, as follows. (For details on inverse limits, see e.g.\ \cite{EngelkingBook} or \cite{ProfiniteGroups}.
Recall that a poset $(P,{\le})$ is said to be \emph{directed} if for all $p,q\in P$ there is an $r\in P$ with $r\ge p$ and $r\ge q$.)
Note that $\cX$ is directed by inclusion; for every $X\in\cX$ let $\cC_X$ consist of the components of $G-X$; endow each $\cC_X$ with the discrete topology; and let $\cb_{X',X}\colon\cC_{X'}\to\cC_X$ for $X'\supset X$ send each component of $G-X'$ to the component of $G-X$ containing it; then $\{\cC_X,\,\cb_{X',X},\,\cX\}$ is an inverse system whose inverse limit, by construction, consists of the directions.

\begin{theorem}[{\cite[Theorem~2.2]{Ends}}]\label{thm: directions}
Let $G$ be any graph. Then the map $\omega\mapsto f_\omega$ is a bijection between the ends of $G$ and its directions, i.e.\ $\Omega=\invLim{}\cC_X$.
\end{theorem}

\noindent From now on we do not distinguish between $\Omega$ and the inverse limit space $\invLim{}\cC_X$ with the inverse limit topology, and we call $\Omega$ the end space.

If a graph $G$ is locally finite, then the star-comb lemma always yields a comb.
This fact has been generalised in Lemma~\ref{KXgiveCompactClosure} below, where the proof relies on the combination of Halin's combinatorial definition of an end with the topological inverse limit point of view on ends as directions:

\begin{lemma}\label{KXgiveCompactClosure}
Let $G$ be any graph and let $U\subset V(G)$ be infinite.
If for every $X\in\cX$ only finitely many components of $G-X$ meet $U$, then $\Abs{U}$ is a non-empty and compact subspace of $\Omega$.
\end{lemma}
\begin{proof}
For every $X\in\cX$ let $\cK_X\subset\cC_X$ consist of the finitely many components of $G-X$ that meet $U$.
Then the closed subspace $\Abs{U}$ of the inverse limit $\Omega=\invLim{}\cC_X$ is non-empty and compact as inverse limit of its non-empty compact Hausdorff projections $\cK_X$, cf.~\cite[Corollary~2.5.7]{EngelkingBook}.
\end{proof}

The combination of topology and infinite graph theory is known as \emph{topological infinite graph theory} (an overview on this young field is presented in~\cite{RDsBanffSurvey,DiestelBook5}).
And in fact, Lemma~\ref{KXgiveCompactClosure} can be employed to deduce a well-known result of Diestel from this field,~\cite[Theorem~4.1]{VTopComp}, which states that a graph is compactified by its ends if and only if it is \emph{tough} in that deleting any finite set of vertices always leaves only finitely many components.
(If $G$ is tough and a covering of $G\sqcup\Omega$ with basic open sets is given, first apply Lemma~\ref{KXgiveCompactClosure} to $V$ to obtain a finite subcover $\cO$ of $\Omega$, then apply Lemma~\ref{KXgiveCompactClosure} to $U=V\setminus\bigcup\cO$ to deduce that $U$ is finite and, therefore, $G\setminus\bigcup\cO$ is compact.)

Since Lemma~\ref{KXgiveCompactClosure} yields combs even when there are both combs and stars (for example if $G$ is an infinite complete graph), this plus of control makes it a useful addition to the star-comb lemma.

\subsection{Critical vertex sets}

We have indicated above that adding the ends generally does not suffice to compactify a graph with the usual topologies. 

However, every graph is naturally compactified by its ends plus critical vertex sets, where a finite set $X$ of vertices of an infinite graph $G$ is \emph{critical} if the collection \[\CX:=\{\,C\in\cC_X\mid N(C)=X\,\}\] is infinite (cf.~\cite{EndsAndTangles,StoneCechTangles,EndsTanglesCrit}).
When $G$ is connected, all its critical vertex sets are non-empty, and so it follows that $G$ having a critical vertex set is stronger than $G$ containing an infinite star: On the one hand, given a critical vertex set $X$, each $x\in X$ sends an edge to each of the infinitely many components $C\in\CX$ and therefore is the centre of an infinite star.
On the other hand, if $G$ is obtained from a ray $R$ by completely joining its first vertex $r$ to all the other vertices of $R$, then $G$ contains an infinite star but no critical vertex set.

Let us say that a critical vertex set $X$ of $G$ lies \emph{in the closure} of $M$ where $M$ is either a subgraph of $G$ or a set of vertices of $G$, if infinitely many components in $\CX$ meet $M$.
The collection of all critical vertex sets of $G$ is denoted by $\crit(G)$.
The \emph{combinatorial remainder} of a graph $G$ is the disjoint union $\Gamma(G):=\Omega(G)\sqcup\crit(G)$.
As usual, $\Gamma=\Gamma(G)$, and $\AbsC{M}$ consists of those $\gamma\in\Gamma$ lying in the closure of $M$.
We obtain a slight strengthening of the star-comb lemma:

\begin{lemma}\label{combinatorialRemainder}
Let $G$ be any graph and let $U\subset V(G)$ be infinite.
Then at least one of the following assertions holds:
\begin{enumerate}
    \item $G$ has an end lying in the closure of $U$;
    \item $G$ has a critical vertex set lying in the closure of $U$.
\end{enumerate}
\end{lemma}
\begin{proof}
If there is a vertex set $X'\in\cX$ such that infinitely many components of $G-X'$ meet $U$, then $X'$ includes a critical vertex set $X$ such that infinitely many components in $\CX$ meet $U$, giving (ii).
Otherwise Lemma~\ref{KXgiveCompactClosure} gives~(i).
\end{proof}


\subsection{Normal trees}

A rooted tree $T\subset G$, not necessarily spanning, is said to be \emph{normal} in $G$ if the endvertices of every $T$-path in $G$ are comparable in the tree-order of $T$, \cite[p.~220]{DiestelBook5}. 
We say that a vertex set $W\subset V(G)$ is \emph{normally spanned} in $G$ if there is a normal tree in $G$ that contains~$W$.
A graph $G$ is \emph{normally spanned} if $V(G)$ is normally spanned, i.e., if $G$ has a normal spanning tree.

The \emph{generalised up-closure} $\guc{x}$ of a vertex $x\in T$ is the union of $\uc{x}$ with the vertex set of $\bigcup \cC(x)$, where the set $\cC(x)$ consists of those components of $G-T$ whose neighbourhoods meet $\uc{x}$.
Every graph $G$ reflects the separation properties of each normal tree $T\subset G$ (we generalise \cite[Lemma~1.5.5]{DiestelBook5} to possibly non-spanning normal trees):

\begin{lemma}\label{NTseparationAndComponents}
Let $G$ be any graph and let $T\subset G$ be any normal tree.
\begin{enumerate}
    \item Any two vertices $x,y\in T$ are separated in $G$ by the vertex set $\dc{x}\cap\dc{y}$.
    \item Let $W\subset V(T)$ be down-closed. Then the components of $G-W$ come in two types: the components that avoid $T$; and the components that meet $T$, which are spanned by the 
    sets $\guc{x}$ with $x$ minimal in $T-W$. 
\end{enumerate}
\end{lemma}

\begin{proof}
(i) The proof is that of~\cite[Lemma~1.5.5~(i)]{DiestelBook5}.

(ii) In a first step, we prove that if a component $C$ of $G-W$ meets $T$ and $x$ is minimal in $C\cap T$, then $C= G[\guc{x}]$. 
The backward inclusion holds because $\guc{x}$ is connected, avoids $W$ and contains $x$. 
The forward inclusion can be seen as follows. On the one hand, $C\cap T\subseteq \uc{x}$. 
Indeed, by (i), any $x$--$y$ path in $C$ with $y\in C\cap T$ contains a vertex below both $x$ and $y$ and every such vertex must be the minimal vertex $x$ itself.
On the other hand, $C- T\subseteq  \bigcup \cC(x)$. Indeed, every component $C'$ of $C- T$ is a component of $G-T$ since $W\subseteq T$, and by $C\cap T\subset\uc{x}$ each neighbour of $C'$ inside $C$ must be contained in $\uc{x}$.

Now let us deduce (ii). 
Without loss of generality $W$ is not empty.
To begin, we prove that each component $C$ of $G-W$ meeting $T$ is spanned by $\guc{x}$ for some minimal $x$ in $T-W$.
By the first step, it suffices to show that a minimal vertex $x$ of $C\cap T$ is also minimal in $T-W$, a fact that we verify as follows. The vertices below $x$ form a chain $\dc{t}$ in $T$. As $t$ is a neighbour of $x$, the maximality of $C$ as a component of $G-W$ implies that $t\in W$, giving $\dc{t}\subset W$ since $W$ is down-closed. Hence $x$ is also minimal in $T-W$.

Conversely, if $x$ is any minimal element of $T-W$, it is clearly also minimal in $C\cap T$ for the component $C$ of $G-W$ to which it belongs.
Together with the first step we conclude that $C$ is a component of $G-W$ meeting $T$ and spanned by~$\guc{x}$.
\end{proof}

As a consequence, the \emph{normal rays} of a normal spanning tree $T\subset G$, those that start at the root, reflect the end structure of $G$ in that every end of $G$ contains exactly one normal ray of $T$, \cite[Lemma~8.2.3]{DiestelBook5}.
More generally,

\begin{lemma}\label{NormalTreeNormalRay}
If $G$ is any graph and $T\subset G$ is any normal tree,
then every end of $G$ in the closure of $T$ contains exactly one normal ray of $T$.
Moreover, sending these ends to the normal rays they contain defines a bijection between $\Abs{T}$ and the normal rays of~$T$.
\end{lemma}

\begin{proof}
Let $\omega$ be any end of $G$ in the closure of $T$.
By Lemma~\ref{NTseparationAndComponents}~(i) at most one normal ray of $T$ is contained in $\omega$, and so it remains to find a normal ray of $T$ that lies in $\omega$.
For this, we pick a comb in $G$ \at $T$ with its spine in $\omega$.
We construct a normal ray of $T$ in $\omega$, as follows.

Starting with the root $v_0$ of $T$, recursively choose nodes $v_0,v_1,v_2,\ldots$ of $T$ such that $v_{n+1}$ is the minimal vertex
of $T-\dc{v_n}$ for which $\guc{v_{n+1}}$ spans
the component of $G-\dc{v_n}$ that contains all but finitely many vertices of the comb. Such a vertex $v_{n+1}$ exists by Lemma~\ref{NTseparationAndComponents} (ii). 
And it is an upward neighbour of $v_n$, which can be seen by applying Lemma~\ref{NTseparationAndComponents} (i) to $v_n$ and $v_{n+1}$. 
In conclusion $v_0v_1v_2\ldots$ is a normal ray of $T$ that is equivalent in $G$ to the spine of the comb. 

The `moreover' part holds as every normal ray of $T$ has its end in $G$ contained in the closure of $T$.
\end{proof}

Consequently, if $G$ contains a comb \at $T$, then $T$ contains exactly one normal ray that is equivalent in $G$ to that comb's spine.

\begin{lemma}\label{NormalTreeNormalCriticalVertexSet}
Let $G$ be any graph and let $T\subset G$ be any normal tree.
Then every critical vertex set of $G$ in the closure of $T$ is contained in $T$ as a chain.
\end{lemma}
\begin{proof}
Let $X$ be any critical vertex set of $G$ that lies in the closure of $T$.
For every component $C\in\CX$ that meets $T$, pick a $C$--$X$ edge from $T$. 
By the pigeonhole principle, some infinitely many of these edges have the same endpoint $x\in X$, giving rise to an infinite star in $T$.
Then, by Lemma~\ref{NTseparationAndComponents}, $\dc{x}$ pairwise separates all the leaves of the star above $x$ at once; let us write $L$ for the set of these leaves.
Since $\dc{x}$ is finite, all but finitely many of the infinitely many components in $\CX$ that meet $L$ are also components of $G-\dc{x}$. And every vertex from $X$ defines at least one path of length two between distinct such components, by the definition of critical vertex sets. Therefore, no vertex in $X$ can be contained in a component of $G-\dc{x}$; in other words, $X$ is contained in the chain $\dc{x}$.
\end{proof}

\subsection{Containing vertex sets cofinally}
Recall that a subset $X$ of a poset $P=(P,{\le})$ is \emph{cofinal} in $P$, and ${\le}$, if for every $p\in P$ there is an $x\in X$ with $x\ge p$.
We say that a rooted tree $T\subset G$ contains a set $W$ \emph{cofinally} if $W\subset V(T)$ and $W$ is cofinal in the tree-order of $T$.
Interestingly, our next lemma does not require $T$ to be normal.

\begin{lemma}\label{NTcombinatorialClosureCofinal}
Let $G$ be any graph.
If $T\subset G$ is a rooted tree that contains a vertex set $W$ cofinally, then $\AbsC{T}=\AbsC{W}$.
\end{lemma}
\begin{proof}
We first prove that $\Abs{T}=\Abs{W}$.
The backward inclusion $\Abs{T}\supseteq \Abs{W}$ holds as $T$ contains $W$. 
For the forward inclusion we prove equivalently that every end of $G$ that is not contained in the closure of $W$ also does not lie in the closure of $T$. 
So consider any end $\omega\in \Omega\setminus \Abs{W}$, and pick a finite vertex set $X\subseteq V(G)$ separating $W$ from $\omega$. 
We claim that the finite set $X'$ consisting of the vertices in $X$ and all vertices in the down-closure of $X\cap V(T)$ in $T$, i.e. $X':=X\cup\lceil X\cap V(T)\rceil_T$, separates $T$ from $\omega$. 
Indeed, suppose for a contradiction that the component $C:=C(X',\omega)$ of $G-X'$ meets $T$. 
Consider a vertex $v\in C\cap T$. 
As $X'\cap V(T)$ is down-closed in $T$, the up-closure $\lfloor v \rfloor_T$ is included in $C$. 
Hence---as $T$ contains $W$ cofinally---the component $C$ also contains a vertex from $W$, contradicting the assumption that $X\subseteq X'$ separates $W$ from $\omega$.

It remains to show that $\AbsC{T}$ and $\AbsC{W}$ coincide on $\crit(G)$.
From $W\subset T$ we infer $\AbsC{W}\subset\AbsC{T}$, so it suffices to show that every critical vertex set that lies in the closure of $T$ does also lie in the closure of $W$.
For this, let any critical vertex set $X\in\AbsC{T}$ be given.
We pick, for every component $C\in\CX$ meeting $T$, a vertex $u(C)$ of $T$ in~$C$.
Then applying the star-comb lemma in $T$ to this infinite vertex set yields either a star or a comb \at it.
Since the finite vertex set $X$ pairwise separates every two vertices in the attachment set at once, we in fact get a star.
Consider the centre of the star.
This is a vertex of $T$ that has infinitely many pairwise incomparable vertices $u(C)$ above it.
Using that $T$ contains $W$ cofinally, we find a vertex $w(C)$ in $T\cap W$ above every~$u(C)$.
As $X$ is finite, we may assume without loss of generality that every vertex $w(C)$ is contained in~$C$.
Then $X$ lies in the closure of the vertex set formed by the vertices $w(C)$, and hence $X\in\AbsC{W}$ follows.
\end{proof}


\subsection{Tree-decompositions and \texorpdfstring{$\boldsymbol{S}$}{S}-trees}

We assume familiarity with \cite[Section~12.3]{DiestelBook5} up to but not including Lemma~12.3.2, and with the concepts of oriented separations and $S$-trees for $S$ a set of separations of a given graph as presented in~\cite[Section~12.5]{DiestelBook5}.
Whenever we introduce a tree-de\-com\-po\-si\-tion as $(T,\cV)$ we tacitly assume that $\cV=(V_t)_{t\in T}$.
Usually we refer to the adhesion sets of a tree-decomposition as separators.
We call a tree-decomposition \emph{rayless} and \emph{locally finite} if the decomposition tree $T$ is rayless and locally finite, respectively.
A \emph{star-decomposition} is a tree-decomposition whose decomposition-tree is a star~$K_{1,\kappa}$ for some cardinal~$\kappa$.
A \emph{rooted} tree-decomposition is a tree-decomposition $(T,\cV)$ where $T$ is rooted.
We say that a rooted tree-decomposition $(T,\cV)$ of $G$ \emph{covers} a vertex set $U\subset V(G)$ \emph{cofinally} if the set of nodes of $T$ whose parts meet $U$ is cofinal in the tree-order of $T$.

We will need the following standard facts about tree-decompositions:

\begin{lemma}[{\cite[Lemma~12.3.1]{DiestelBook5}}]\label{TDCseps}
Let $G$ be any graph with a tree-decomposition $(T,\cV)$ and let $t_1t_2$ be any edge of $T$ and let $T_1,T_2$ be the components of $T-t_1t_2$, with $t_1\in T_1$ and $t_2\in T_2$.
Then $V_{t_1}\cap V_{t_2}$ separates $A_1:=\bigcup_{t\in T_1}V_t$ from $A_2:=\bigcup_{t\in T_2}V_t$ in $G$.
\end{lemma}

\begin{corollary}\label{connectedSubgraphTDC}
Let $(T,\cV)$ be any tree-decomposition of any graph $G$.
If a connected subgraph $H\subset G$ avoids a part $V_t$, then there is a unique component $T'$ of $T-t$ with $H\subset\bigcup_{t'\in T'}G[V_{t'}]$ and $H$ avoids every part that is not at a node of the component~$T'$.\qed
\end{corollary}

A tree-decomposition $(T,\cV)$ makes $T$ into an $S$-tree for the set $S$ of separations it induces, cf.~\cite{DiestelBook5}.
The converse is true, for example if $T$ is rayless, but false in general (it is no longer clear that every vertex of $G$ lives in some part if $T$ contains a ray).
By a simple distance argument, however, the converse holds in a special case for which we need the following definition.
Suppose that $(T,\alpha)$ is an $S$-tree with $T$ rooted in $r\in T$.
We say that the separators of $(T,\alpha)$ are \emph{upwards disjoint} if for every two edges $\ve<\vf$ pointing away from the root $r$ the separators of $\alpha(\ve)$ and $\alpha(\vf)$ are disjoint.
Here, $\ve=(e,s,t)$ \emph{points away from}~$r$ if $r\le _T s<_T t$, i.e., if $s\in rTt$.
Then every $S$-tree with upwards disjoint separators induces a tree-de\-com\-po\-si\-tion.

We use the following non-standard notation for $S$-trees $(T,\alpha)$: for an edge $xy=e$ of the decomposition tree $T$ we abbreviate $\alpha(e,x,y)=\alpha(x,y)$.

\subsection{Tree-decompositions and \texorpdfstring{$\boldsymbol{S}$}{S}-trees displaying sets of ends}\label{subsec:TDCsDisplayDef}

In this section we give a brief summary of how the ends of $G$ relate to the decomposition trees of tree-decompositions and $S$-trees. 
For the sake of readability, we introduce all needed concepts for $S$-trees and let the tree-decompositions inherit these concepts from their corresponding $S$-trees.

Let $(T,\alpha)$ be any $\Sinf$-tree.
If $\omega$ is an end of $G$, then $\omega$ orients every finite-order separation $\{A,B\}\in\Sinf$ of $G$ towards the side $K\in \{A,B\}$ for which every ray in $\omega$ has a tail in $G[K]$.
In this way, $\omega$ induces a consistent orientation of $\vSinf$ and, via $\alpha$, also induces a consistent orientation $O$ of $\vE(T)$.
Then $\omega$ either \emph{lives} at a unique node $t\in T$ in that the star $\vF_{\! t}=\{\,(e,s,t)\in\vE(T)\mid e=st\in T\,\}$ at $t$ is included in $O$, or \emph{corresponds} naturally to a unique end $\eta$ of $T$ in that for some (equivalently:\ every) ray $t_1t_2\ldots$ in $\eta$ all oriented edges $(t_n t_{n+1}, t_n, t_{n+1})$ are contained in $O$.
When $(T,\alpha)$ corresponds to a tree-decomposition $(T,\cV)$ and $\omega$ lives at $t$, then we also say that $\omega$ \emph{lives} in the part $V_t$ at $t$.
Moreover, we remark that $\omega$ lives in $V_t$ if and only if some (equivalently:\ every) ray in $\omega$ has infinitely many vertices in $V_t$.
Likewise, $\omega$ corresponds to $\eta$ if and only if some (equivalently:\ every) ray $R\in\omega$ follows the course of some (equivalently:\ every) ray $W\in\eta$ (in that for every tail $W'\subset W$ the ray $R$ has infinitely many vertices in $\bigcup_{t\in W'}V_t$).
In both cases `having infinitely many vertices in' cannot be replaced with `having a tail in', e.g.\ consider decomposition trees that are infinite stars or combs whose teeth avoid their spines.

Consider the map $\tau\colon\Omega(G)\to\Omega(T)\sqcup V(T)$ that takes each end of $G$ to the end or node of $T$ which it corresponds to or lives at respectively.
This map essentially captures how the ends of $G$ relate to the ends of $T$. 
We say that $(T,\alpha)$ \emph{displays} a set of ends $\Psi\subset\Omega(G)$ if $\tau$ restricts to a bijection $\tau\rest\Psi\colon\Psi\to\Omega(T)$ between $\Psi$ and the end space of $T$  and maps every end that is not contained in $\Psi$ to some node of~$T$.

It is a natural and largely open question for which subsets $\Psi\subset\Omega(G)$ a graph $G$ has a tree-decomposition $(T,\cV)$ that displays $\Psi$.
Only recently, Carmesin achieved a major breakthrough by providing a positive answer for $\Psi$ the set of undominated ends of $G$.
In order to state his result in its full strength, we introduce two more definitions and motivate them in a lemma.

Suppose that $T$ is rooted in $r\in T$, that $G$ is connected and that the separators of $(T,\alpha)$ are non-empty. 
If the finite separators of $(T,\alpha)$ are upwards disjoint, then by the star-comb lemma and a simple distance argument, every end of $T$ has some ends of $G$ corresponding to it (i.e.\ $\tau^{-1}(\eta)\neq\emptyset$ for every end $\eta$ of $T$).
And if additionally $(T,\alpha)$ is \emph{upwards connected} in that for every edge $\ve$ pointing away from the root $r$ the induced subgraph $G[B]$ stemming from $(A,B)=\alpha(\ve)$ is connected, then $T$ already displays the set of those ends of $G$ that correspond naturally to ends of $T$ (i.e.\ $\vert\tau^{-1}(\eta)\vert=1$ for every end $\eta$ of $T$):

\begin{lemma}\label{DisplayLemma}
Let $G$ be any graph.
Every upwards connected rooted $\Sinf$-tree $(T,\alpha)$ with upwards disjoint non-empty separators displays the ends of $G$ that correspond to the ends of $T$.
\end{lemma}

\begin{proof}
By our preliminary remarks it remains to show that for every end $\eta$ of $T$ there is at most one end of $G$ corresponding to~$\eta$.
Suppose for a contradiction that $\eta$ is an end of $T$ such that two distinct ends $\omega\neq\omega'$ of $G$ correspond to it, and write $R$ for the rooted ray of $T$ that represents $\eta$.
Pick $X\in\cX$ such that $\omega$ and $\omega'$ live in distinct components of $G-X$.
As the separators of $(T,\alpha)$ are upwards disjoint, by a distance argument we find an edge $e\in R$ with orientation $\ve$ away from the root such that the separation $(A,B)=\alpha(\ve\,)$ satisfies $B\cap X=\emptyset$.
Now both of the two ends $\omega$ and $\omega'$ have rays in $G[B]$ because both of them correspond to $\eta$.
And in $G[B]$ we find paths connecting these rays, since $(T,\alpha)$ is upwards connected.
But then these rays and paths avoid $X$, contradicting the choice of~$X$.
\end{proof}

Now we are ready to state the following result of Carmesin~\cite{carmesin2014all} that solved a conjecture of Diestel~\cite{EndStructureResultsOpenProblems} from 1992 (in amended form) and, as a corollary, also solved a conjectured of Halin~\cite{halin64} from 1964 (again in amended form):

\begin{theorem}[{\cite{carmesin2014all}}]\label{CarmesinsTDCforTopologicalEnds}
Every connected graph $G$ has a rooted tree-decomposition with upwards disjoint finite connected separators that displays the undominated ends~of~$G$.
\end{theorem}

\noindent The theorem above combines Carmesin's Theorem~1, Remark~6.6, the second paragraph of his `Proof that Theorem~1 implies Corollary~2.6', and a standard argument to make the separators connected.



\section{Combs}\label{section:CombTreeCrit}

\noindent Jung~\cite{jung69} noted that, given any connected graph $G$ and any vertex set $U\subset V(G)$, the absence of a comb \at $U$ is equivalent to $U$ being \emph{dispersed} in $G$, meaning that for every ray $R\subset G$ there is a finite vertex set $X\subset V(G)$ separating $R$ from $U$. 
This equivalence then gives another equivalence as $U$ being dispersed rephrases to `no end of $G$ lies in the closure of $U$'.
For readers familiar with the topological space $\vert G\vert=G\sqcup\Omega$ as in~\cite{DiestelBook5}, this is to say that $U$ is closed in~$\vert G\vert$.
These assertions---while equivalent to the absence of a comb---are abstract and do not immediately provide concrete structures that are complementary to combs.
Providing concrete complementary structures is the aim of this section.


\subsection{Normal trees}\label{subsection:CombNT}

In this section we prove
\begin{mainresult}\label{CombTreeDualityI} 
\TFAD
\begin{enumerate}
\item $G$ contains a comb \at $U$;
\item there is a rayless normal tree $T\subset G$ that contains $U$.
\end{enumerate}
Moreover, the normal tree $T$ in \emph{(ii)} can be chosen such that it contains $U$ cofinally.
\end{mainresult}

\noindent For this, we need the following key results of Jung's proof of his 1967 characterisation, Theorem~\ref{JungTheorem}, of the connected graphs that have normal spanning trees.

\begin{proposition}[Jung]\label{JungsRaylessSequence}
Let $G$ be any connected graph and let $U\subset V(G)$ be any vertex set.
If $U$ is a countable union $\bigcup_{n\in\N}U_n$ of dispersed sets $U_n\subset V(G)$ and $v$ is any vertex of $G$, then $G$ contains an ascending sequence $T_0\subset T_1\subset\cdots$ of rayless normal trees $T_n\subset G$ such that each $T_n$ contains $U_0\cup\cdots\cup U_n$ cofinally and is rooted in $v$.
In particular, the overall union $T:=\bigcup_{n\in\N}T_n$ is a normal tree in $G$ that contains $U$ cofinally and is rooted in $v$.
\end{proposition}
\begin{proof}
It suffices to show that, given a rayless normal tree $T_n$ containing $U_{\le n}:=U_0\cup\cdots\cup U_n$ cofinally, we find a rayless normal tree $T_{n+1}$ extending $T_n$ and containing $U_{\le n+1}=U_{\le n}\cup U_{n+1}$ cofinally.
For this, let any $T_n$ be given.
Consider the collection of all normal trees $T\supset T_n$ with $T\cap U_{\le n+1}$ cofinal in the tree-order of $T$, partially ordered by letting $T\le T'$ whenever $T$ is extended by $T'$ as a normal tree.
Since $U_{n+1}$ is dispersed and $T_n$ is rayless, all of these trees must be rayless.
Let $T_{n+1}$ be a maximal tree that Zorn's lemma provides for this poset. In the following we show that $T_{n+1}$ is as desired.

Assume for a contradiction that some vertex $u\in U_{\le n+1}$ is not contained in $T_{n+1}$.
Since $T_{n+1}$ is normal, the neighbourhood of the component $C$ of $G-T_{n+1}$ that contains $u$ forms a chain in the tree-order of $T_{n+1}$.
As $T_{n+1}$ is rayless, this chain has a maximal node $x\in T_{n+1}$.
Let $T'$ be the union of $T_{n+1}$ and an $x$--$u$ path $P$ with $\mathring{P}\subset C$.
Then the neighbourhood in $T'$ of any new component $C'\subset C$ of $G-T'$ is a chain in $T'$, so $T'$ is again normal.
But then $T'$ contradicts the maximality of $T_{n+1}$, completing the proof that $T_{n+1}$ is as desired.
\end{proof}

\begin{corollary}[Jung]\label{JungCoverDispersedSet}
Let $G$ be any graph and let $U\subset V(G)$ be any vertex set.
If $U$ is dispersed itself and $v$ is any vertex of $G$, then $G$ contains a rayless normal tree that contains $U$ cofinally and is rooted in $v$.\qed
\end{corollary}

\begin{corollary}[Jung]\label{JungCountableSet}
Let $G$ be any graph and let $U\subset V(G)$ be any vertex set.
If $U$ is countable and $v$ is any vertex of $G$, then $G$ contains a normal tree that contains $U$ cofinally and is rooted in $v$.\qed
\end{corollary}

\begin{lemma}\label{NTlevelsDispersed}
Let $G$ be any graph.
The vertex set of any rayless normal tree $T\subset G$ is dispersed.
In particular, the levels of any normal tree $T\subset G$ are dispersed.
\end{lemma}
\begin{proof}
Lemma~\ref{NormalTreeNormalRay}.
\end{proof}

Jung's abstract characterisation of the normally spanned graphs goes as follows: 
\begin{theorem}[{Jung,~\cite[Satz~6]{jung69}}]\label{JungTheorem}
Let $G$ be any graph. A vertex set $W\subset V(G)$ is normally spanned in $G$ if and only if it is a countable union of dispersed sets. In particular, 
$G$ is normally spanned if and only if $V(G)$ is a countable union of dispersed sets.
\end{theorem}
For an excluded-minor characterisation of the connected graphs with normal spanning trees see \cite{DiestelLeaderNST,PitzNewNSTobstructions}.

\begin{proof}[{Proof of Theorem~\ref{JungTheorem}}]
The backward implication is provided by Proposition~\ref{JungsRaylessSequence}.
The forward implication holds as the levels of any normal tree are dispersed, Lemma~\ref{NTlevelsDispersed}.
\end{proof}

We are now ready to prove Theorem~\ref{CombTreeDualityI}:

\begin{proof}[Proof of Theorem~\ref{CombTreeDualityI}]
First, to show that at most one of (i) and (ii) holds, we show (ii)$\to\neg$(i).
If $T\subset G$ is a rayless normal tree containing $U$, then $V(T)$ is dispersed by Lemma~\ref{NTlevelsDispersed}, and hence so is~$U\subset V(T)$.

It remains to show that at least one of (i) and (ii) holds; we show $\neg$(i)$\to$(ii).
Since the absence of a comb with all its teeth in $U$ means that $U$ is dispersed, Corollary~\ref{JungCoverDispersedSet} yields a rayless normal tree in $G$ that contains $U$ cofinally.
\end{proof}

\subsection{Tree-decompositions}\label{subsection:CombTree}

In this section, we show how the rayless normal tree from Theorem~\ref{CombTreeDualityI} gives rise to a tree-decomposition that is complementary to combs.

\begin{mainresult}\label{CombTreeDualityII}
\TFAD\
\begin{enumerate}
\item $G$ contains a comb \at $U$;
\item $G$ has a rayless tree-decomposition into parts each containing at most finitely many vertices from $U$ and whose parts at non-leaves of the decomposition tree are all finite.
\end{enumerate}
Moreover, the rayless tree-decomposition in \emph{(ii)} displays $\Abs{U}$ and may be chosen with connected separators.
\end{mainresult}

We start with a lemma which shows that at most one of (i) and (ii) holds.

\begin{lemma}\label{CombTree_StarComb_Crit}
Let $G$ be any graph and let $U\subset V(G)$ be any vertex set.
Suppose that $G$ has a rayless tree-decomposition into parts each containing at most finitely many vertices from $U$ and whose parts at non-leaves of the decomposition tree are all finite.
Then for every infinite $U'\subset U$ there is a critical vertex set of $G$ that lies in the closure of $U'$.
\end{lemma}

\begin{proof}
Let such a tree-decomposition $(T,\cV)$ of $G$ be given for $U$, and let $U'$ be an arbitrary infinite subset of $U$.
For every $u\in U'$ we choose a node $t_u\in T$ with $u\in V_{t_u}$.
Since each part of the tree-decomposition contains at most finitely many vertices from $U$, we may assume without loss of generality (moving to an infinite subset of $U'$) that the nodes $t_u$ are pairwise distinct.
Hence applying Lemma~\ref{sclForRaylessTrees} in the rayless tree $T$ yields a star $S$ \at $\{\,t_u\mid u\in U'\,\}$.
Without loss of generality (as before) we may assume that the nodes $t_u$ form precisely the attachment set of $S$ and that no vertex $u$ from $U'$ is contained in the finite part $V_c$ at the central node $c$ of $S\subset T$.
For every $u\in U'$ let $C_u$ be the component of $G-V_c$ containing $u$.
Then distinct vertices from $U'$ are contained in distinct components of $G-V_c$ by Lemma~\ref{TDCseps}.
Since the finite part $V_c$ contains the neighbourhood of each component $C_u$, by the pigeon-hole principle we find a subset $X\subset V_c$ which is precisely equal to the neighbourhood of $C_u$ for some infinitely many $u\in U'$.
\end{proof}

\begin{proof}[Proof of Theorem~\ref{CombTreeDualityII}]
By Lemma~\ref{CombTree_StarComb_Crit} at most one of (i) and (ii) holds.
It remains to show that at least one of (i) and (ii) holds.

We show $\neg$(i)$\to$(ii).
Let $\nt\subset G$ be a rayless normal tree containing $U$ as provided by Theorem~\ref{CombTreeDualityI}.
We construct the desired tree-decomposition from $\nt$.
As $\nt$ is rayless and normal, the neighbourhood of any component $C$ of $G-\nt$ is a finite chain in the tree-order of $\nt$, and hence has a maximal element $t_C\in \nt$.
Now, let the tree $T$ be obtained from $\nt$ by adding each component $C$ of $G-\nt$ as a new vertex and joining it precisely to $t_C$.
The tree $T$ will be our decomposition tree; it remains to name the parts.
For nodes $t\in \nt\subset T$ we let $V_t$ consist of the down-closure $\lceil t\rceil_{\nt}$ of $t$ in the normal tree $\nt$.
And for newly added nodes $C\in T- \nt$ we let $V_C$ be the union of $V_{t_C}$ and the vertex set of the component~$C$, i.e., we put $V_C=\lceil t_C\rceil_{\nt}\cup V(C)$.
It is straightforward to check that $T$ with these parts forms a tree-decomposition of $G$ that meets the requirements of (ii) and satisfies the theorem's `moreover' part. 
\end{proof}

Our next example shows that Theorem~\ref{CombTreeDualityII}~(ii) cannot be strengthened so as to get a star as decomposition tree or to have pairwise disjoint separators:
\begin{example}\label{CombStarDualNo}
Suppose that $G$ consists of the first three levels of $T_{\aleph_0}$, the tree all whose vertices have countably infinite degree, and let $U=V(G)$.
Then $G$ is rayless so there is no comb \at $U$.

First, $G$ has no star-decomposition into parts each containing at most finitely many vertices from $U$: Indeed, assume for a contradiction that $G$ has such a star-decomposition $(S,\cV)$, and let $c$ be the centre of the infinite star $S$.
As the part $V_c$ contains at most finitely many vertices from $U=V(G)$ it must be finite.
Then each component of $G-V_c$ is contained in some $G[V_\ell]$ with $\ell$ a leaf of $S$ by Corollary~\ref{connectedSubgraphTDC}.
As each part of $(S,\cV)$ contains at most finitely many vertices from $U$, this means that every component of $G-V_c$ contains at most finitely many vertices from $U=V(G)$ and hence is finite.
But as $V_c$ is finite, $G-V_c$ must have an infinite component, a contradiction.

Second, $G$ also has no rayless tree-decomposition with finite and pairwise disjoint separators such that each part contains at most finitely many vertices from $U$: Indeed, suppose for a contradiction that $G$ has such a tree-decomposition $(T,\cV)$.
Without loss of generality we may assume that all its parts are non-empty.
The rayless decomposition tree $T$ has a vertex $t$ of infinite degree, so $V_t$ contains infinitely many of the finite and pairwise disjoint separators.
As $G$ is connected, all of these are non-empty by Lemma~\ref{TDCseps}, so $V_t$ is infinite, and hence so is $V_t\cap U=V_t$.
But this contradicts our assumptions.
\end{example}


\subsection{Critical vertex sets}\label{subsection:CombCrit}
The absence of a comb \at $U$ is equivalent to $U$ being dispersed, which is to say that no end of $G$ lies in the closure of $U$.
With the combinatorial remainder $\Gamma(G)=\Omega(G)\sqcup\crit(G)$ compactifying $G$ in mind, this means that only critical vertex sets of $G$ lie in the closure of $U$, i.e.\  $\AbsC{U}\subset\crit(G)$.
Phrasing this combinatorially gives

\begin{mainresult}\label{CombCritDuality} 
\TFAD
\begin{enumerate}
\item $G$ contains a comb \at $U$;
\item for every infinite $U'\subset U$ there is a critical vertex set $X\subset V(G)$ such that infinitely many of the components in $\CX$ meet $U'$.
\end{enumerate}
\end{mainresult}

Quantifying over all $U'$ in Theorem~\ref{CombCritDuality} is necessary for (ii)$\to\neg$(i), e.g.,  if $G$ is an infinite star of rays with $U=V(G)$. 
We remark that Theorem~\ref{CombCritDuality} implies Halin's \cite[Satz~1]{Halin1965} from 1965 which reads as follows:
\emph{A graph $G$ is rayless if and only if every infinite $M\subset V(G)$ has an infinite subset $M'$ for which there is a finite $H\subset G$ such that every component of $G-H$ contains only finitely many vertices of~$M'$.}

Since, by now, the right tools are at hand, we can prove Theorem~\ref{CombCritDuality} in two efficient ways: 

\begin{proof}[Combinatorial proof of Theorem~\ref{CombCritDuality} using Theorem~\ref{CombTreeDualityI} or~\ref{CombTreeDualityII}]
Clearly, at most one of (i) and (ii) can hold.
And if $G$ contains no comb \at $U$, then (ii) holds by Theorem~\ref{CombTreeDualityI} with Lemma~\ref{sclForRaylessTrees} or by Theorem~\ref{CombTreeDualityII} with Lemma~\ref{CombTree_StarComb_Crit}.
\end{proof}

\begin{proof}[Inverse limit proof of Theorem~\ref{CombCritDuality}]
Lemma~\ref{combinatorialRemainder} gives $\neg$(i)$\to$(ii).
\end{proof}
Note that condition (ii) yields a star \at $U$.


\subsection{Rank}\label{subsection:CombRank}

In 1983, Schmidt~\cite{Schmidt1983} introduced a notion that is now known as the \emph{Schmidt rank} or just \emph{rank} of a graph, cf.~Chapter~8.5 of~\cite{DiestelBook5}.
His rank provides a recursive definition of the class of rayless graphs which enables us to prove assertions about rayless graphs by transfinite induction.
An outstanding application of this technique is the proof of the unfriendly partition conjecture for rayless graphs, cf.~\cite{UnfriendlyPartition,DiestelBook5}.
Since the absence of a comb \at $U$ is equivalent to the existence of a \emph{rayless} normal tree containing $U$, Theorem~\ref{CombTreeDualityI}, one may wonder whether there somehow is a link to the Schmidt rank.
In this section we show that this is indeed the case.

Schmidt defines the rank of a graph as follows.
He assigns \emph{rank} 0 to all the finite graphs.
And given an ordinal $\alpha>0$, he assigns \emph{rank} $\alpha$ to every (not necessarily connected) graph $G$ that does not already have a rank $\beta<\alpha$ and which has a finite set $X$ of vertices such that every component of $G-X$ has some rank $<\alpha$.

\begin{lemma}[{\cite{Schmidt1983}}]
Let $G$ be any graph. Then the following assertions are complementary:
\begin{enumerate}
    \item $G$ contains a ray;
    \item $G$ has a rank.
\end{enumerate}
\end{lemma}

Now we introduce the notion of a $U$-\emph{rank}, based on the Schmidt rank, which additionally takes into account a fixed set $U$.
For this, suppose that $U$ is any set.
Even though, formally, $U$ is an arbitrary set, we think of $U$ as a set of vertices.
Let us assign $U$-\emph{rank} 0 to all the graphs that contain at most finitely many vertices from~$U$.
Given an ordinal $\alpha>0$, we assign $U$-\emph{rank} $\alpha$ to every graph $G$ that does not already have a $U$-rank $\beta<\alpha$ and which has a finite set $X$ of vertices such that every component of $G-X$ has some $U$-rank $<\alpha$.
Note that the rank of $G$ is equal to the $V$-rank of~$G$.

The $U$-rank behaves quite similarly to the Schmidt rank, \cite[p.~243]{DiestelBook5}:
When disjoint graphs $G_i$ have $U$-ranks $\alpha_i<\alpha$, their union clearly has a $U$-rank of at most $\alpha$; if the union is finite, it has $U$-rank $\max_i\alpha_i$.
Induction on $\alpha$ shows that subgraphs of graphs of $U$-rank $\alpha$ also have a $U$-rank of at most $\alpha$.
Conversely, joining finitely many new vertices to a graph, no matter how, will not change its $U$-rank.

Not every graph has a $U$-rank.
Indeed, a comb \at $U$ cannot have a $U$-rank, since deleting finitely many of its vertices always leaves a component that is a comb \at $U$.
As subgraphs of graphs with a $U$-rank also have a $U$-rank, this means that only graphs without such combs can have a $U$-rank.
But all these~do:

\begin{mainresult}\label{CombRankDuality} 
Let $G$ be any graph and let $U$ be any set.
Then \tfad
\begin{enumerate}
    \item $G$ contains a comb \at $U$;
    \item $G$ has a $U$-rank.
\end{enumerate}
\end{mainresult}

\noindent Phrased differently, the $U$-rank provides a recursive definition of the class of the graphs in which $U$ is dispersed.

\begin{proof}[Proof of Theorem~\ref{CombRankDuality}]
We show the equivalence (i)$\,{\leftrightarrow}\,\neg$(ii).
The forward implication has already been pointed out above.
For the backward implication suppose that $G$ has no $U$-rank; we show that $G$ must contain a comb \at $U$.
As $G$ has no $U$-rank, one of its components, $C_0$ say, has no $U$-rank as well.
Pick $u_0\in U\cap C_0$ arbitrarily.
Since $C_0$ has no $U$-rank, it follows that $C_0-u_0$ has a component $C_1$ that has no $U$-rank; let $u_1\in U\cap C_1$ and pick a $u_0$--$u_1$ path $P_1$ in $C_0$ with $\mathring{P}_1\subset C_1$.
Next, delete $P_1$ from $C_1$ and let $C_2\subset C_1-P_1$ be a component that has no $U$-rank.
Let $u_2\in U\cap C_2$, pick any $P_1$--$u_2$ path $P_2$ in $C_1$ with $\mathring{P}_2\subset C_2$ and note that $P_2$ meets $P_1$ in $\mathring{u}_1P_1$.
Therefore, if we continue inductively to find paths $P_1,P_2,\ldots$ in~$G$, then their union $\bigcup_n P_n$ is a comb with attachment set $\{\,u_n\mid n\in\N\,\}\subset U$.
\end{proof}

There is a way to see immediately that for a connected graph $G$ having a $U$-rank is stronger than $G$ containing a star \at $U$ when $U$ is infinite.
For this, suppose that $G$ has $U$-rank $\alpha$.
Then $\alpha>0$ as $U\subset V(G)$ is infinite.
Hence $G$ has a finite set $X$ of vertices such that every component of $G-X$ has some $U$-rank~$<\alpha$.
In particular, $G-X$ must have some infinitely many components that meet~$U$.
Each of these components gives some $U$--$X$ path avoiding all other components, so the pigeon-hole principle yields a star \at $U$ as desired.

The $U$-rank of a graph has many properties.
In the remainder of this section, we prove three such properties that we will put to use in the next section.

\begin{lemma}\label{UrankAlgebra}
Let $G$ be any graph, let $U$ be any set and suppose that $G$ has $U$-rank~$\alpha$.
Then the following assertions hold:
\begin{enumerate}
    \item for every subset $U'\subset U$ the graph $G$ has $U'$-rank $\le\alpha$;
    \item for every subgraph $H\subset G$ the graph $H$ has $U$-rank $\le\alpha$.
\end{enumerate}
\end{lemma}
\begin{proof}
Induction on $\alpha$.
\end{proof}

\begin{lemma}\label{raylessTreeRankIsUrank}
Let $U$ be any set.
If $T$ is a rooted rayless tree containing $U\cap V(T)$ cofinally, then the $U$-rank of $T$ is equal to the rank of $T$.
\end{lemma}

\noindent Here we remark that, in this paper, we consider the Schmidt rank of rayless graphs as discussed in Section~\ref{subsection:CombRank}.
In particular, when we consider the rank of a (possibly rooted) tree, we do not mean the rank for rooted trees that defines recursive prunability (cf.~\cite[p.~242 \&~243]{DiestelBook5}).

\begin{proof}[{Proof of Lemma~\ref{raylessTreeRankIsUrank}}]
Let $\alpha$ be the $U$-rank of $T$ and let $\beta$ be the rank of $T$. 
Since the $V(T)$-rank of $T$ is the same as the rank of $T$, Lemma~\ref{UrankAlgebra}~(i) gives the inequality $\alpha \le \beta$. 
An induction on $\alpha$ shows the converse inequality (in the induction step consider a set $X\subseteq V(T)$ witnessing that $T$ has $U$-rank $\alpha$ and employ the induction hypothesis to see that every component of $T-X$ has rank $<\alpha$; it is convenient to assume $X$ to be down-closed, which is possible by Lemma~\ref{UrankAlgebra}~(ii)).
\end{proof}

\newpage
\begin{lemma}\label{UrankNormalTree}
If $G$ is any graph and $T\subset G$ is a rayless normal tree containing $U\cap G$ cofinally, then the following three ordinals are all equal:
\begin{enumerate}
    \item the rank of $T$;
    \item the $U$-rank of $T$;
    \item the $U$-rank of $G$.
\end{enumerate}
\end{lemma}
\begin{proof}
The equality (i)${}={}$(ii) is the subject of Lemma~\ref{raylessTreeRankIsUrank}.
Lemma~\ref{UrankAlgebra} gives the inequality (ii)${}\le{}$(iii).
We show the remaining inequality (iii)${}\le{}$(ii) by induction on the $U$-rank of $T$, as follows.

If the $U$-rank of $T$ is $0$, then $U\cap T=U\cap G$ is finite, and thus the $U$-rank of $G$ is $0$ as well.
For the induction step, suppose that $T$ has $U$-rank $\alpha> 0$, and let $X\subset V(T)$ be any finite vertex set such that every component of $T-X$ has $U$-rank~$<\alpha$.
By Lemma~\ref{UrankAlgebra}~(ii) we may assume that $X$ is down-closed in $T$.
It suffices to show that every component of $G-X$ has a $U$-rank $<\alpha$.

If $C$ is a component of $G-X$, then either $C$ avoids $T\supset U\cap C$ and has $U$-rank $0<\alpha$, or $C$ meets $T$.
In the case that $C$ meets $T$, by Lemma~\ref{NTseparationAndComponents} we know that $C$ is spanned by $\guc{x}$ with $x$ minimal in $T-X$, so $T\cap C\subset C$ is a normal tree containing $U\cap C$ cofinally.
Finally, by the induction hypothesis,
\[
(U\text{-rank of }C)\le (U\text{-rank of }T\cap C)<\alpha.\qedhere
\]
\end{proof}

\subsection{Combining the duality theorems}\label{subsection:CombStrongDual}

So far we have seen duality theorems for combs in terms of
normal trees, tree-decompositions, critical vertex sets and rank.
With these four complementary structures for combs at hand, the question arises whether it is possible to combine them all.
In this section we will answer the question in the affirmative.
That is, we will present a fifth complementary structure for combs that combines all of the four~above.

This fifth structure will be a `tame' tree-decomposition that is more specific than the tree-decomposition listed above.
It will stem from a normal tree in a way that we call `squeezed expansion'. 
Just like the tree-decomposition listed above, all its parts will meet $U$ finitely, and all its parts at non-leaves will be finite.
Moreover, it will display not only the ends in the closure of $U$, but also the critical vertex sets in the closure of~$U$.
In order to realise this, we will extend the definition of `display' in a reasonable way.
Finally, the decomposition tree will have a rank that is equal to the $U$-rank of the whole graph.
The combined duality theorem reads as follows:

\begin{mainresult}\label{CombStrongDual}
\TFAD
\begin{enumerate}
    \item $G$ contains a comb \at $U$;
    \item $G$ has a rooted tame tree-decomposition $(T,\cV)$ that covers $U$ cofinally and satisfies the following four assertions:
    \begin{itemize}
        \item[\textbf{--}] $(T,\cV)$ is the squeezed expansion of a normal tree in $G$ that contains the vertex set~$U$ cofinally;
        \item[\textbf{--}] every part of $(T,\cV)$ meets $U$ finitely and parts at non-leaves are finite;
        \item[\textbf{--}] $(T,\cV)$ displays $\AbsC{U}\subset\crit(G)$;
        \item[\textbf{--}] the rank of $T$ is equal to the $U$-rank of $G$.
    \end{itemize}
\end{enumerate}
\end{mainresult}

\begin{corollary}
If a connected graph $G$ is rayless (equivalently:\ if $G$ has a rank), then $G$ has a tame tree-decomposition into finite parts that displays the combinatorial remainder of $G$ and has a decomposition tree whose rank is equal to the rank of~$G$.\qed
\end{corollary}

The proof of the theorem above is organised as follows. 
First, we will state Proposition~\ref{ExpansionProperties}, which lists some useful properties of squeezed expansions.  
Then, we will employ this proposition in a high level proof of Theorem~\ref{CombStrongDual}. 
In order to follow the line of argumentation up to here, it is not necessary to know the definitions of `tame', `display' and `squeezed' `expansion', which is why we will introduce them subsequently to our high level proof. Finally, we will prove Proposition~\ref{ExpansionProperties}.

\begin{proposition}\label{ExpansionProperties}
Let $G$ be any graph and suppose that $\nt\subset G$ is a normal tree such that every component of $G-\nt$ has finite neighbourhood, that $(T,\cV)$ is the expansion of $\nt$ and that $(T',\cW)$ is a squeezed $(T,\cV)$.
Then the following assertions hold: 
\begin{enumerate}
    \item $(T,\cV)$ is upwards connected;
    \item both $(T,\cV)$ and $(T',\cW)$ display $\AbsC{\nt}$ (in particular, both are tame);
    \item all the parts of $(T,\cV)$ and $(T',\cW)$ meet $\nt$ finitely;
    \item parts of $(T',\cW)$ at non-leaves of $T'$ are finite;
    \item $T'$ is rayless if and only if $T$ is rayless if and only if $\nt$ is rayless;
    \item if one of $T'$, $T$ and $\nt$ is rayless, then the ranks of $T'$, $T$ and $\nt$ all exist and are all equal.
\end{enumerate}
\end{proposition}

The proposition has a corollary that is immediate because every normal spanning tree will have an expansion, and expansions will be rooted:

\begin{corollary}
Every normally spanned graph has a rooted tame tree-de\-com\-po\-si\-tion displaying its combinatorial remainder.\qed
\end{corollary}

Now we prove Theorem~\ref{CombStrongDual} using Proposition~\ref{ExpansionProperties} above:

\begin{proof}[Proof of Theorem~\ref{CombStrongDual}]
(i) and (ii) exclude each other for various reasons we have already discussed.

For the implication $\neg$(i)$\to$(ii) suppose that $G$ contains no comb \at ~$U$.
By Theorem~\ref{CombTreeDualityI} there is a rayless normal tree $\nt\subset G$ that contains $U$ cofinally.
We show that the squeezed expansion $(T',\cW)$ of $\nt$ is as desired.
By Proposition~\ref{ExpansionProperties} every part of $(T',\cW)$ meets $\nt\supset U$ finitely and parts at non-leaves of $T'$ are finite.
As we have $\AbsC{\nt}=\AbsC{U}$ by Lemma~\ref{NTcombinatorialClosureCofinal}, Proposition~\ref{ExpansionProperties} also ensures that the squeezed expansion $(T',\cW)$ of $\nt$ displays $\AbsC{U}$ (in particular, $(T',\cW)$ is tame).
Finally, the $U$-rank of $G$ exists by Theorem~\ref{CombRankDuality} and is equal to the rank of $\nt$ by Lemma~\ref{UrankNormalTree}, which in turn is equal to the rank of $T'$ by Proposition~\ref{ExpansionProperties}.
\end{proof}

Next, we provide all the definitions needed:
First, we define `tame' tree-de\-com\-po\-si\-tions (Definition~\ref{definition:tame}).
Second, we extend the definition of `display' to include critical vertex sets (Definition~\ref{def: display extended}).
Third, we define the `expansion' of a normal tree (Definition~\ref{def: expansion}), which is a certain tree-decomposition. 
Finally we define what it means to `squeeze' a tree-decomposition (Definition~\ref{def:squeezed}).

Recall that the definition of `display', as discussed in Section~\ref{section:preliminaries}, highly relies on the fact that the ends of a graph orient all its finite-order separations. 
Now, critical vertex sets are closely related to ends, as they together with the ends turn graphs into compact topological spaces. 
This is why we may hope that every critical vertex set $X$ orients the finite-order separations so as to lead immediately to a notion of `displaying a collection of critical vertex sets'. 
Probably the most natural way that a critical vertex set $X$ could orient a finite-order separation $\{A,B\}$ towards a side $K\in\{A,B\}$ is that $X$ together with all but finitely many of the components in $\CX$ are contained in $K$. 

However, this is too much to ask: For example consider an infinite star. 
The centre $c$ of the star forms a critical vertex set $X=\{c\}$, and any separation with separator $X$ that has infinitely many leaves on both sides will not be oriented by $X$ in this way. 

But focusing on a suitable class of separations, those that are \emph{tame}, leads to a natural extension of `display' to include critical vertex sets:

\begin{definition}\label{definition:tame}
A finite-order separation $\sep{X}{\cC}$ 
of $G$ is \emph{tame} if for no $Y\subset X$ both $\cC$ and $\cC_X\setminus\cC$ contain infinitely many components whose neighbourhoods are precisely equal to~$Y$.
\end{definition}

\noindent The tame separations of $G$ are precisely those finite-order separations of $G$ that respect the critical vertex sets:

\begin{lemma}\label{tameAndCritical}
A finite-order separation $\{A,B\}$ of a graph $G$ is tame if and only if every critical vertex set $X$ of $G$ together with all but finitely many components from $\CX$ is contained in one side of $\{A,B\}$.
\end{lemma}
\begin{proof}
For the forward implication, note that every distinct two vertices of a critical vertex set are linked in $G[X\cup\bigcup\CX]$ by infinitely many independent paths, so every critical vertex set of $G$ meets at most one component of $G-(A\cap B)$.
\end{proof}

We say that an $\Sinf$-tree $(T,\alpha)$ is \emph{tame} if all the separations in the image of $\alpha$ are tame.
And we say that a tree-decomposition is \emph{tame} if it corresponds to a tame $\Sinf$-tree.

If $X$ is a critical vertex set of $G$ and $(T,\alpha)$ is a tame $\Sinf$-tree, then $X$ induces a consistent orientation of the image of $\alpha$ by orienting every tame finite-order separation $\{A,B\}$ towards the side that contains $X$ and all but finitely many of the components from $\CX$ (cf.~Lemma~\ref{tameAndCritical} above).
This consistent orientation also induces a consistent orientation of $\vE(T)$ via $\alpha$.
Then, just like for ends, the critical vertex set $X$ either \emph{lives} at a unique node $t\in T$ or \emph{corresponds} to a unique end of $T$.
In this way, we obtain an extension $\sigma\colon\Gamma(G)\to\Omega(T)\sqcup V(T)$ of the map $\tau\colon\Omega(G)\to\Omega(T)\sqcup V(T)$ from Section~\ref{subsec:TDCsDisplayDef}.

Since $\sigma$ extends $\tau$ from the end space $\Omega(G)$ of $G$ to the full combinatorial remainder $\Gamma(G)$ of $G$, it is reasonable to wonder why the target set of $\sigma$ is that of~$\tau$, namely $\Omega(T)\sqcup V(T)$, rather than analogously taking the target set $\Gamma(T)\sqcup V(T)$.
At a closer look, the critical vertex sets of $T$ are already contained in the target set $\Omega(T)\sqcup V(T)$, for they are precisely the infinite degree nodes of $T$.
This, and the fact that every critical vertex set $X$ of $G$ naturally comes with an oriented tame separation $\crsep{X}$ of $G$, motivate the following definition.

\begin{definition}\label{def: display extended}[Display $\Psi\subset\Gamma(G)$]
Let $G$ be any graph.
A rooted tame $\Sinf$-tree $(T,\alpha)$ \emph{displays} a subset $\Psi$ of the combinatorial remainder $\Gamma(G)=\Omega(G)\sqcup\crit(G)$ of $G$ if $\sigma$ satisfies the following three conditions:
\begin{itemize}\itemsep.2em
    \item $\sigma$ restricts to a bijection between $\Psi\cap\Omega(G)$ and $\Omega(T)$;
    \item $\sigma$ restricts to a bijection between $\Psi\cap\crit(G)$ and the infinite-degree nodes of $T$ so that:
    whenever $\sigma$ sends a critical vertex set $X\in \Psi$ to $t\in T$, then $t$ has a predecessor $s\in T$ with $\alpha(s,t)=\rsep{X}{\cC}$ such that $\cC\subset\CX$ is cofinite and $\alpha$ restricts to a bijection between $\vF_{\! t}$ and the star in $\vSinf$ that consists of the separation $\rsep{X}{\cC}$ and all the separations $\lsep{C}{X}$ with $C\in\cC$; 
    \item $\sigma$ sends all the elements of $\Gamma(G)\setminus\Psi$ to finite-degree nodes of $T$.
\end{itemize}
\end{definition}

Note that this definition of displays is not exactly an extension of the original definition given in Section~\ref{subsec:TDCsDisplayDef}.
Indeed, if $(T,\alpha)$ displays $\Psi$ and $\omega\notin \Psi$ is an end, then with the original definition $\omega$ may correspond to an infinite degree vertex of~$T$, but not with the new definition.
However, the new definition is stronger than the original one: if $(T,\alpha)$ displays $\Psi\subset\Gamma(G)$ in the new sense, then $(T,\alpha)$ displays $\Psi\cap\Omega(G)$ in the original sense.

We solve this ambiguity as follows.
Whenever we say that a tree-decomposition or $\Sinf$-tree displays some set $\Psi$ of ends of $G$ and it is clearly understood that we view $\Psi$ as a subset of $\Omega(G)$, e.g.\ when we let $\Psi$ consist of the undominated ends of $G$ or consider $\Psi=\Abs{U}$, then by `displays' we refer to the original definition from Section~\ref{subsec:TDCsDisplayDef}.
But whenever we explicitly introduce $\Psi$ as a subset of the combinatorial remainder $\Gamma(G)$ of $G$, e.g.\ when we let $\Psi$ consist of critical vertex sets or consider $\Psi=\AbsC{U}$, then by `displays' we refer to the new definition introduced above.

We wish to make a few remarks on our new definition.
If $(T,\alpha)$ is a rooted tame $\Sinf$-tree displaying some $\Psi\subset\Gamma(G)$ and the tree-decomposition $(T,\cV)$ corresponding to $(T,\alpha)$ exists, then $V_{\sigma(X)}=X$ whenever $X$ is a critical vertex set in $\Psi$.
We do not require $\cC=\CX$ in the definition of displays because there are simply structured normally spanned graphs for which otherwise none of their tree-decompositions would display their combinatorial remainder.
See \cite[Examples~3.6~\&~3.7]{DistinguishUltrafilterTangles} for details.

Now, let us turn to the expansion of a normal tree.
Given vertex sets $Y\subset X\subset V(G)$ we write $\cC_X(Y)$ for the collection of all components $C\in\cC_X$ with $N(C)=Y$.

\begin{definition}[Expansion of a normal tree]\label{def: expansion}
In order to define the expansion, suppose that $G$ is any connected graph and $\nt\subset G$ is any normal tree such that every component of $G-\nt$ has finite neighbourhood.
From the normal tree $\nt$ we obtain the tame \emph{expansion} $(T,\cV)$ of $\nt$ in $G$ in two steps, as follows.

For the first step, let us suppose without loss of generality that for all nodes $t\in \nt$ every up-neighbour $t'$ of $t$ in $\nt$ is named as the component $\guc{t'}$ of $G-\dc{t}$ containing~$t'$.
(Formally, we realise this as labelling of the vertex set to avoid a conflict with the axiom of foundation.)
We define a map $\beta\colon\vE(\nt)\to\vSinf$ by letting $\beta(t,C):=\rsep{N(C)}{C}$ and $\beta(C,t):=\beta(t,C)^\ast$ whenever $C$ is an up-neighbour of a node $t$ in $\nt$.
Then $(\nt,\beta)$ is a rooted tame $\Sinf$-tree that displays $\Abs{\nt}\subset\Omega(G)$.

In the second step, we obtain from $(\nt,\beta)$ a rooted tame $\Sinf$-tree $(T,\alpha)$ displaying $\AbsC{\nt}\subset\Gamma(G)$. 
Informally speaking we sort the separations of the form $\beta(t,C)$ with $t\in \nt$ an infinite degree-node and $C$ an up-neighbour of $t$ in $\nt$ by the critical vertex sets $X\subseteq \dc{t}$ in the closure of $\nt$ with $C\in \CX$. Formally this is done as follows (cf.~Figure~\ref{fig:Expansion}).

\begin{figure}[ht]
\centering
\def\svgwidth{\columnwidth}
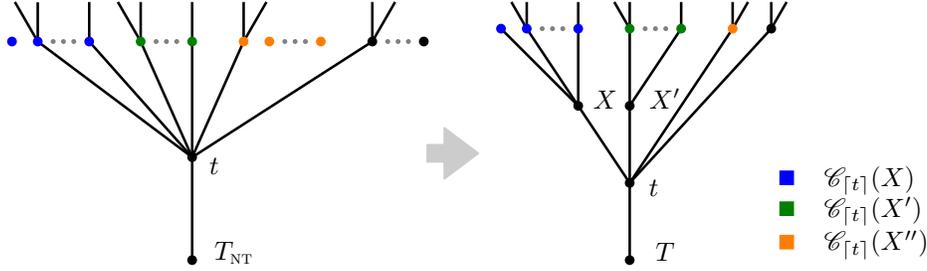
\caption{The second step in the construction of the expansion of normal trees. The critical vertex sets $X$ and $X'$ are in the closure of $\nt$, while $X''$ is not.
The three sets $X$, $X'$ and $X''$ are all the critical vertex sets of $G$ that contain $t$ and are contained in~$\dc{t}$.}
\label{fig:Expansion}
\end{figure} 

For every infinite-degree node $t\in \nt$ and every critical vertex set $X\in\AbsC{\nt}$ satisfying $t\in X\subset\dc{t}$ we do the following: 
\begin{enumerate}
    \item we add a new vertex named $X$ to $\nt$ and join it to $t$;
    \item for every component $C\in\cC_{\dc{t}}(X)\subset\CX$ we delete the edge $tC$ (this is redundant when $\nt$ avoids $C$) and add the new edge $XC$ (note that in particular the vertex $C$ gets added as well, even if $\nt$ avoids $C$);
    \item we let $\alpha(t,X):=\rsep{X}{\cC_{\dc{t}}(X)}$, and for every component $C\in \cC_{\dc{t}}(X)$ we let $\alpha(X,C):=\rsep{X}{C}$.
\end{enumerate}
Then we take $T$ to be the resulting tree, and we extend $\alpha$ to all of $\vE(T)$ by letting $\alpha(\ve):=\beta(\ve)$ whenever the edge $e$ of $T$ is also an edge of the normal tree $\nt$.
The rooted tame tree-decomposition $(T,\cV)$ corresponding to $(T,\alpha)$ is the \emph{expansion} of $\nt$ in $G$.\xqed{$\diamondsuit$}
\end{definition}

And here is the definition of squeezing:

\begin{definition}[Squeezing a tree-decomposition]\label{def:squeezed}
Suppose that $(T,\cV)$ and $(T',\cW)$ are tree-decompositions of $G$.
We say that $(T',\cW)$ is a \emph{squeezed} $(T,\cV)$ if $(T',\cW)$ is obtained from $(T,\cV)$ as follows.
The tree $T'$ is obtained from $T$ by adding, for every node $t\in T$ that has finite degree $>1$ and whose part $V_t$ is infinite, a new node $t'$ to $T$ and joining it to $t$.
For all these nodes $t$ the part $W_t$ is the union of the separators of $(T,\cV)$ associated with the edges of $T$ at $t$, and the part $W_{t'}$ is taken to be the part $V_t$.
For all other nodes $t$ the part $W_t$ is $V_t$.
\end{definition}

Note that if $(T',\cW)$ is the squeezed $(T,\cV)$ and all separators of $(T,\cV)$ are finite, then all the infinite parts $V_t$ with $t$ an internal finite-degree node of $T$ become finite parts $W_t$.
Thus, all parts $W_t$ with $t$ an internal finite-degree node of $T'$ are finite.
Achieving this property is the purpose of squeezing.

Squeezing preserves tameness:

\begin{lemma}\label{lem: squeezingPreservesTame}
Let $G$ be any graph, let $(T,\cV)$ be any tree-de\-com\-po\-si\-tion of $G$ with finite separators and let $(T',\cW)$ be the squeezed $(T,\cV)$.
If $(T,\cV)$ is tame, then $(T',\cW)$ is tame as well.
\end{lemma}

\noindent To prove this lemma, we need the following fact about tame separations.
(In~the language of separation systems, it states that the tame separations form a subuniverse of the universe of finite-order separations of a graph.)

\begin{lemma}\label{lem: tame subuniverse}
If $(A_0,B_0),\ldots,(A_n,B_n)$ are tame separations of a graph $G$ for some $n\in\N$, then the separation $(\bigcup_{i\le n}A_i,\bigcap_{i\le n}B_i)$ is tame as well.
\end{lemma}
\begin{proof}
The assertion follows by induction on~$n$ once we have shown the case~\mbox{$n=1$}.
By Lemma~\ref{tameAndCritical} it suffices to find for every critical vertex set $X$ of $G$ a cofinite subset $\cC\subset\CX$ such that $X\cup V[\cC]\subset A_0\cup A_1$ or $X\cup V[\cC]\subset B_0\cap B_1$.
Given~$X$, let $\cC_0$ and $\cC_1$ be two cofinite subsets of~$\CX$ such that $X\cup V[\cC_0]$ is included in either $A_0$ or $B_0$, and $X\cup V[\cC_1]$ is included in either $A_1$ or $B_1$.
We are done with $\cC=\cC_0$ or $\cC=\cC_1$ except possibly when $X\cup V[\cC_i]\subset B_i$ for both $i=0,1$.
In this case, $\cC:=\cC_0\cap\cC_1$ is a cofinite subset of~$\CX$ with $X\cup V[\cC]\subset B_0\cap B_1$.
\end{proof}

\begin{proof}[Proof of Lemma~\ref{lem: squeezingPreservesTame}]
Suppose that $(T,\cV)$ is a tame tree-decomposition of $G$ and that $(T',\cW)$ is the squeezed $(T,\cV)$.
Separations of $G$ that are induced by $(T',\cW)$ are tame when they are induced by edges of $T'$ that are also edges of $T\subset T'$.
Hence it suffices to show that for every leaf $\ell\in T'-T$ with neighbour $t\in T\subset T'$ the separation induced by $\ell t\in T'$ is tame.
For this, let any edge $\ell t\in T'$ be given and write $s_0,\ldots, s_n$ for the finitely many neighbours of $t$ in $T$.
Let $(T',\alpha')$ be the $\Sinf$-tree corresponding to $(T',\cW)$, let $(A,B):=\alpha'(t,\ell)$ and define $(A_i,B_i):=\alpha'(s_i,t)$ for all $i\le n$.
Then, by the definition of $(T',\cW)$, we have $A=\bigcup_i A_i$ and $B=\bigcap_i B_i$.
Since all separations $(A_i,B_i)$ are tame, it follows by Lemma~\ref{lem: tame subuniverse} that $(A,B)$ is tame as well.
%
\end{proof}

Now that we have formally introduced all the definitions involved, we are ready to prove Proposition~\ref{ExpansionProperties}:

\begin{proof}[{Proof of Proposition~\ref{ExpansionProperties}}]
(i) The expansion is upwards connected by definition.

(ii) Using Lemma~\ref{NormalTreeNormalRay} and the fact that every component of $G-\nt$ has finite neighbourhood, it is straightforward to check that $(T,\cV)$ displays $\Abs{\nt}\subset\Omega(G)$.
We verify that $(T,\cV)$ even displays $\AbsC{\nt}\subset\Gamma(G)$.
On the one hand, by Lemma~\ref{NormalTreeNormalCriticalVertexSet} every critical vertex set $X\in\AbsC{\nt}$ is contained in $\nt$ as a chain, and hence appears precisely once as a node of $T$ by the definition of the expansion.
On the other hand, every node of infinite degree of $T$ stems from such a critical vertex set.
Together we obtain that $(T,\cV)$ displays $\AbsC{\nt}$.
The tree-decomposition $(T',\cW)$ is tame because $(T,\cV)$ is, cf.~Lemma~\ref{lem: squeezingPreservesTame}.
From here, it is straightforward to show that $(T',\cW)$ displays $\AbsC{\nt}$ as well.

(iii) and (iv) are straightforward.

(v) follows from (ii) and Lemma~\ref{NormalTreeNormalRay}.

(vi) It is straightforward to check by induction on the rank that the rank is preserved under taking contraction minors with finite branch sets. Similarly, one can show that two infinite trees have the same rank if one is  obtained from the other by adding new leaves to some of its nodes of infinite degree.
Now we deduce (vi) as follows.
For every node $t\in \nt$ let us write $S_t$ for the finite star with centre $t$ and leaves the critical vertex sets $X\in\AbsC{\nt}$ with $t\in X\subset\dc{t}$.
The decomposition tree $T$ of the expansion of $\nt$ is obtained from an $I\nt\subset T$ with finite branch sets (the non-trivial branch sets are precisely the vertex sets of the stars $S_t$ for the nodes $t\in\nt$ of infinite degree) by adding leaves to  nodes of infinite degree (each leaf is a component $C\in\cC_{\dc{t}}(X)$ avoiding $\nt$ for some $X\in S_t$ and gets joined to $X\in I\nt\subset T$).
Therefore, the ranks of $T$ and $\nt$ coincide.
The decomposition tree $T'$ is obtained from $T$ by adding at most one new leaf to each node of $T$, and new leaves are only added to finite-degree nodes of $T$.
An induction on the rank shows that the rank is preserved under this operation, and so the ranks of $T'$ and $T$ coincide as well.
\end{proof}

Carmesin~\cite{carmesin2014all} showed that every connected graph $G$ has a tree-decomposition with  finite separators that displays $\Psi$ for $\Psi$ the set of undominated ends of $G$, cf.~Theorem~\ref{CarmesinsTDCforTopologicalEnds}.
He then asked for a characterisation of those pairs of a graph $G$ and a subset $\Psi\subset\Omega(G)$ for which $G$ has such a tree-decomposition displaying $\Psi$.
In the same spirit, our findings motivate the following problem:

\begin{problem}
Characterise, for all connected graphs $G$, the subsets $\Psi\subset\Gamma(G)$ for which $G$ admits a rooted tame tree-decomposition displaying $\Psi$.
\end{problem}

\section{Stars}\label{section:StarRay}

\subsection{Normal trees}\label{subsection:StarNT}

In this section we prove a duality theorem for stars in terms of normal trees.

\begin{mainresult}\label{InitialStarNTduality}
\TFAD
\begin{enumerate}
    \item $G$ contains a star \at $U$;
    \item there is a locally finite normal tree $T\subset G$ that contains $U$  and all whose rays are undominated in $G$.
\end{enumerate}
Moreover, the normal tree $T$ in \emph{(ii)} can be chosen such that it contains $U$ cofinally and every component of $G-T$ has finite neighbourhood.
\end{mainresult}

\begin{proof}[{Proof of Theorem~\ref{InitialStarNTduality} without the `moreover' part}]
First, we show that at most one of (i) and (ii) holds.
Assume for a contradiction that both hold.
Let $T\subset G$ be a normal tree as in (ii) and let $U'\subset U$ form the attachment set of some star \at $U$.
By Lemma~\ref{sclForLocFinTrees} the locally finite tree $T$ contains a comb \at $U'$.
That comb's spine, then, is dominated in $G$ by the centre of the star, a contradiction.

It remains to show that at least one of (i) and (ii) holds; we show $\neg$(i)$\to$(ii).
We have that $U$ is countable, since otherwise the star-comb lemma yields a star \at $U$.
By Corollary~\ref{JungCountableSet} we find a normal tree $T\subset G$ that contains $U$ cofinally.
Clearly, $T$ must be locally finite since $G$ contains no star \at $U$.
For the same reason, every ray of $T$ is undominated in $G$. 
\end{proof}

The remaining `moreover' part is a consequence of \cite[Theorem~1]{StarComb2TheDominatedComb}, which is why its proof is placed in the second paper of our series, cf.~\cite[Section~2]{StarComb2TheDominatedComb}.
To see immediately that a locally finite normal tree $T$ as in (ii) is more specific than a comb when $U$ is infinite, apply Lemma~\ref{sclForLocFinTrees} to $T$.


\subsection{Tree-decompositions}\label{subsection:StarTree}

For combs we have provided a duality theorem in terms of normal trees, and that theorem then gave rise to another duality theorem in terms of tree-decompositions.
Since we have shown a duality theorem for stars in terms of normal trees in the previous section, a natural question to ask is whether this theorem gives rise to a duality theorem for stars in terms of tree-decompositions, just like for combs.
It turns out that stars have a duality theorem in terms of tree-decompositions.
But this theorem cannot be proved by imitating the proof of the respective theorem for combs, and so we will have to come up with a whole new strategy.
Our theorem reads as follows:

\begin{mainresult}\label{StarRayDualtiy} 
\TFAD
\begin{enumerate}
\item $G$ contains a star \at $U$;
\item $G$ has a locally finite tree-decomposition with finite and pairwise disjoint separators such that each part contains at most finitely many vertices of $U$.
\end{enumerate}
Moreover, the tree-decomposition in \emph{(ii)} can be chosen with connected separators and such that it displays $\AbsC{U}$ which consists only of ends.
\end{mainresult}

\noindent We remark that (ii) is equivalent to the assertion that `$G$ has a ray-decomposition with finite and pairwise disjoint separators such that each part contains at most finitely many vertices of $U$' since the distance classes of locally finite trees are finite.


To prove the theorem, we start by showing that (i) and (ii) exclude each other:

\begin{lemma}\label{StarRayDuality_AtMostOne}
In Theorem~\ref{StarRayDualtiy} the graph $G$ cannot satisfy both \emph{(i)} and \emph{(ii)}.
\end{lemma}
\begin{proof}
Let $(T,\cV)$ be a tree-decomposition as in (ii) of Theorem~\ref{StarRayDualtiy}.
Assume for a contradiction that $G$ contains a star $S$ \at $U$.
As the separators of $(T,\cV)$ are pairwise disjoint, the centre $c$ of $S$ is contained in at most two parts of $(T,\cV)$.
Let $T'\subset T$ be the finite subtree induced by the nodes of these parts plus their neighbours in $T$.
As the parts at the nodes of $T'$ altogether contain at most finitely many vertices from $U$, the star $S$ must send infinitely many paths to vertices in parts at $T-T'$.
But the centre $c$ is separated from the parts at $T-T'$ by the finite union of the finite separators associated with the edges of $T$ leaving $T'$, a contradiction.
\end{proof}

Now, to prove Theorem~\ref{StarRayDualtiy} it remains to show $\neg$(i)$\to$(ii).
This time, however, it is harder to see how the normal tree from Theorem~\ref{InitialStarNTduality} can be employed to yield a tree-decomposition. 
That is why we do not take the detour via normal trees and instead construct the tree-decomposition directly. 
Still, this requires some effort.

First of all, assuming the absence of a star as in (i), we need a strategy to construct a tree-decomposition as in (ii).
Fortunately, we do not have to start from scratch.
In the proof of \cite[Theorem~2.2]{Ends}, Diestel and Kühn proved the following as a technical key result:
\emph{If $\omega$ is an undominated end of $G$, then there exists a sequence $(X_n)_{n\in\N}$ of non-empty finite vertex sets $X_n\subset V(G)$ such that, for all $n\in\N$, the component $C(X_n,\omega)$ contains $X_{n+1}\cup C(X_{n+1},\omega)$.}
Now if $\Abs{U}$ is a singleton $\{\omega\}$, then $\omega$ must be undominated as (i) fails, and we consider such a sequence $(X_n)_{n\in\N}$.
By making all the $X_{n+1}$ connected in $C(X_n,\omega)$ first, and then moving to a suitable subsequence, we obtain a ray-decomposition of $G$ that meets the requirements of (ii).
Our strategy is to generalise this fundamental observation using that $\Abs{U}$ is compact in our situation:

\begin{lemma}\label{onlyUndomEnds}\label{UhatCompact}
If $G$ contains no star \at ~$U$, then $\Abs{U}$ is non-empty, compact and contains only undominated ends.
\end{lemma}
\begin{proof}
By the pigeonhole principle, for every $X\in\cX$ only finitely many components of $G-X$ may meet $U$.
Thus $\Abs{U}$ is non-empty and compact by Lemma~\ref{KXgiveCompactClosure}.
\end{proof}

Our next lemma generalises the fact that a vertex can be strictly separated from every end which it does not dominate.

\begin{lemma}\label{ExhaustingSequenceGeneralised}
Suppose that $X$ is a finite set of vertices in a (possibly disconnected) graph $G$ such that $G-X$ is connected, and that $\Psi\subset\Omega(G)$ is a non-empty and compact subspace consisting only of undominated ends.
Then there is a finite-order separation of $G$ that strictly separates $X$ from $\Psi$ and whose separator is connected.
\end{lemma}
\begin{proof}
No end in $\Psi$ is dominated and $X$ is finite, so for every end $\omega\in\Psi$ we find a finite vertex set $Y(\omega)\subset V(G)$ with $Y(\omega)\cup C(Y(\omega),\omega)$ disjoint from $X$.
Since the components $C(Y(\omega),\omega)$ induce a covering of $\Psi$ by open sets, the compactness of $\Psi$ yields finitely many ends $\omega_1,\ldots,\omega_n\in\Psi$ such that every end in $\Psi$ lives in at least one of the components $C(Y(\omega_i),\omega_i)$.
Let the vertex set $Y$ be obtained from the finite union of the finite sets $Y(\omega_i)$ by adding some finitely many vertices from the connected subgraph $G-X$ so as to ensure that $G[Y]$ is connected.
Note that $Y$ avoids $X$, and write $\cD$ for the collection of the components of $G-Y$ in which ends of $\Psi$ live.
We claim that $\rsep{Y}{\cD}$ strictly separates $X$ from $\Psi$.
For this, let $\omega$ be any end in $\Psi$.
Pick an index $k$ for which $\omega$ lives in the component $C(Y(\omega_k),\omega_k)=:C$.
Then, by the choice of $Y(\omega_k)$, there is no $X$--$C$ path in $G-Y(\omega_k)$.
By $Y(\omega_k)\subset Y$ and $C(Y,\omega)\subset C$ then there certainly is no $X$--$C(Y,\omega)$ path in $G-Y$.
Therefore, $\rsep{Y}{\cD}$ strictly separates $X$ from $\Psi$.
\end{proof}

\begin{proposition}\label{StreeConstruction}
Let $G$ be any connected graph and suppose that $\Psi\subset\Omega$ is a non-empty and compact subspace that consists only of undominated ends.
Then there exists a locally finite $\Sinf$-tree $(T,\alpha)$ with connected pairwise disjoint separators that displays $\Psi$.
\end{proposition}
\begin{proof}
We inductively construct a sequence $\big((T_n,\alpha_n)\big)_{n\in\N}$ of rooted $\Sinf$-trees with root $r\in T_0\subset T_1\subset\cdots$ and $\alpha_0\subset\alpha_1\subset\cdots$, as follows.

To define $(T_0,\alpha_0)$, let $T_0$ consist of one edge $rt$ and put 
$\alpha_0(r,t):=(\{v\},V)$ for an arbitrary vertex $v$ of $G$.
Now, to obtain $(T_{n+1},\alpha_{n+1})$ from $(T_n,\alpha_n)$, we do the following for every edge $t\ell$ of $T_n$ at a leaf $\ell\neq r$.
Consider the separation $\alpha(t,\ell)=\rsep{X}{\cC}$ with $C_1,\ldots,C_n$ the finitely many components in $\cC$ in which ends of $\Psi$ live (these are finitely many as $\Psi$ is compact).
For each component $C_i$ apply Lemma~\ref{ExhaustingSequenceGeneralised} in $G[X\cup C_i]$ to $X$ and $\Psi\cap\Abs{C_i}$ to obtain a finite-order separation $(A_i,B_i)$ of $G[X\cup C_i]$ that strictly separates $X$ from $\Psi\cap\Abs{C_i}$ in $G[X\cup C_i]$ and has a connected separator $A_i\cap B_i$.
Then $(A_i',B_i')$ with $A_i':=A_i\cup (V\setminus C_i)$ and $B_i':=B_i$ is a finite-order separation of $G$ that strictly separates $X$ from $\Psi\cap\Abs{C_i}$ in~$G$ and has a connected separator $A_i'\cap B_i'=A_i\cap B_i$.
We add each $C_i$ as a new node to~$T_n$, join it precisely to the leaf $\ell$ and let $\alpha_{n+1}(\ell,C_i):=(A_i',B_i')$.
This completes the description of our construction.

We claim that the pair $(T,\alpha)$ given by $T:=\bigcup_n T_n$ and $\alpha:=\bigcup_n\alpha_n$ is as required.
Our construction ensures that $T$ is locally finite and that the separators of $(T,\alpha)$ are connected and pairwise disjoint.
Furthermore, our construction ensures that every end in $\Psi$ corresponds to an end of $T$.
It remains to show that $(T,\alpha)$ displays $\Psi$.
By Lemma~\ref{DisplayLemma} it suffices to show that, for every end of $T$, there is an end in $\Psi$ corresponding to it.
And indeed, every ray in $T$ avoiding the root is, literally, a descending sequence $C_1\supset C_2\supset\cdots$ of components for which some end of the compact $\Psi$ lives in all $C_n$ by the finite intersection property of the collection \mbox{$\{\,\Psi\cap\Abs{C_n}\mid n\in\N\,\}$}.
\end{proof}

\begin{proof}[Proof of Theorem~\ref{StarRayDualtiy}]
By Lemma~\ref{StarRayDuality_AtMostOne} at most one of (i) and (ii) can hold.
To establish that at least one of them holds, we show $\neg$(i)$\to$(ii).
Suppose that $G$ contains no star \at $U$.
By Lemma~\ref{onlyUndomEnds} we know that the subspace $\Abs{U}\subset\Omega$ consisting of the ends lying in the closure of $U$ actually contains only undominated ones, and is both non-empty and compact.
Proposition~\ref{StreeConstruction} then yields a locally finite $\Sinf$-tree $(T,\alpha)$ with connected pairwise disjoint separators that displays $\Abs{U}$. Let $(T,\cV)$ be the tree-decomposition corresponding to $(T,\alpha)$.
As $G$ contains no star \at $U$, there is no critical vertex set in the closure of $U$, and hence $(T,\cV)$ even displays~$\AbsC{U}$.
It remains to show that each part of $(T,\cV)$ contains at most finitely many vertices from $U$.
Suppose for a contradiction that some part $V_t$ contains some infinitely many vertices from~$U$, and write $U'$ for that subset of $U$.
As (i) fails, applying Lemma~\ref{UhatCompact} in $G$ to $U'$ yields an end in $\Abs{U'}$.
But then this end lies in $\Psi$ but does not correspond to an end of $T$, a contradiction.
\end{proof}

\bibliographystyle{amsplain}
\bibliography{StarCombBib}
\end{document}